\documentclass[10pt,bezier]{article}
\usepackage{amsmath,amssymb,amsfonts,graphicx,caption,subcaption}
\usepackage[colorlinks,linkcolor=red,citecolor=blue]{hyperref}

\newtheorem{prethm}{{\bf Theorem}}[section]
\newenvironment{thm}{\begin{prethm}{\hspace{-0.5
em}{\bf.}}}{\end{prethm}}
\newtheorem{prepro}{{\bf Theorem}}

\newtheorem{precor}[prethm]{{\bf Corollary}}
\newenvironment{cor}{\begin{precor}{\hspace{-0.5
em}{\bf.}}}{\end{precor}}
\newtheorem{preconj}[prethm]{{\bf Conjecture}}
\newenvironment{conj}{\begin{preconj}{\hspace{-0.5
em}{\bf.}}}{\end{preconj}}
\newtheorem{preremark}[prethm]{{\bf Remark}}

\newtheorem{prelem}[prethm]{{\bf Lemma}}
\newenvironment{lem}{\begin{prelem}{\hspace{-0.5
em}{\bf.}}}{\end{prelem}}
\newtheorem{preque}[prethm]{{\bf Problem}}
\newenvironment{prob}{\begin{preque}{\hspace{-0.5
em}{\bf.}}}{\end{preque}}

\newtheorem{prealphthm}{{\bf Problem}}

\newenvironment{alphprob}
{
\begin{prealphthm}{\hspace{-0.5 em}
{\bf\ }}}{
\end{prealphthm}
}

\newtheorem{preobserv}[prethm]{{\bf Observation}}
\newenvironment{observ}{\begin{preobserv}{\hspace{-0.5
em}{\bf.}}}{\end{preobserv}}

\newtheorem{predef}[prethm]{{\bf Definition}}

\newtheorem{preproposition}[prethm]{{\bf Proposition}}

\newtheorem{preproof}{{\bf Proof.}}
\newtheorem{preprooff}{{\bf Proof}}

\newenvironment{proof}[1]{\begin{preproof}{\rm
#1}\hfill{$\Box$}}{\end{preproof}}

\newtheorem{preproofF}{{\bf Proof of}}

\title{\bf\Large 
The List Square Coloring Conjecture fails for bipartite planar graphs and their line graphs
}

\author{
Morteza Hasanvand\thanks{Department of Mathematical Sciences, Sharif University of Technology, Tehran, Iran.
E-mail: {\tt morteza.hasanvand@alum.sharif.edu}.
Yokohama National University, Yokohama, Japan.
}\\{\footnotesize{${}$\it Dedicated to Faramarz Yaghoobi on the occasion of his $72$th birthday}}
}

\date{}

\begin{document}
\maketitle
\begin{abstract}{
Kostochka and Woodall (2001) conjectured that the square of every graph has the same chromatic number and list chromatic number. In 2015 Kim and Park disproved this conjecture for non-bipartite and bipartite graphs. It was asked by several authors whether this conjecture holds for bipartite graphs with small degrees, claw-free graphs, or line graphs. In this paper, we introduce several kinds of counterexamples to this conjecture to solve three open problems posed by Kim and Park~(2015), Kim, Kwon, and Park~(2015), and Dai, Wang, Yang, and Yu~(2018). In particular, we disprove a planar version of this conjecture proposed by Havet, Heuvel, McDiarmid, and Reed (2017). 

This conjecture was originally proposed to make a stronger version of the List Total Coloring Conjecture. 
In order to make a revised version, it remains to decide whether this conjecture holds for bipartite graphs $G$ by imposing 
a lower bound on the chromatic number of the square graph $G^2$ in terms of its maximum degree as the condition $\chi(G^2) \ge \frac{1}{2} \Delta(G^2)+1$ (or by adding an upper bound on the number of colors used in lists for a weaker version). To support this version, we will show that the bipartite condition cannot be dropped even by increasing the lower bound arbitrarily.

Finally, we investigate non-choosable graphs with bounded maximum degree in bipartite or planar graphs. Consequently, we improve several graph constructions due to Erd\H os, Rubin, and Taylor~(1980), Bessy, Havet, and Palaysi (2002), Voigt (1993), Mirzakhani (1996), and Glebov, Kostochka, and Tashkinov (2005) in terms of maximum degree or order. In addition, we characterize edge-minimal $3$-chromatic non-$3$-choosable (resp. $4$-chromatic non-$4$-choosable) graphs of order at most $9$ (resp. $11$) and settle a question posed by Nelsen~(2019).
\\
\\
\noindent {\small {\it Keywords}: Square graph; list coloring; chromatic-choosable; maximum degree; claw-free; bipartite; planar. }} {\small
}
\end{abstract}
%
%
%
%
%
%
%
%
%
%
\section{Introduction}
In this article, all graphs are considered simple unless otherwise stated. Let $G$ be a graph. 
The vertex set and the edge set of $G$ are denoted by $V(G)$ and $E(G)$, respectively. 
The graph $G$ is said to be {\bf $k$-colorable} if its vertices can be colored by $k$ colors such that adjacent vertices have different colors.
The {\bf chromatic number} $\chi(G)$ of $G$ is the minimum number of such integers $k$.
The graph $G$ is said to be {\bf $k$-choosable} if 
its vertices can be colored such that the color of every vertex $v$ lies in $L(v)$, where $L(v)$ is an arbitrary set of colors with size $k$. The {\bf list chromatic number} $\chi_\ell(G)$ of $G$ is the minimum number of such integers $k$.
A graph $G$ is called {\bf chromatic-choosable}, if $\chi(G) = \chi_\ell(G) $.
We say that a graph $G$ is {\bf $i$-strongly $k$-choosable} if 
its vertices can be colored such that the color of every vertex $v$ lies in $L(v)$, where $L(v)$ is an arbitrary set of colors with size $k$ and the union of all of them has size at most $k+i$.
The {\bf $i$-strong list chromatic number} $\chi^i_\ell(G)$ of $G$ is the minimum number of such integers $k$.
Note that every $i$-strongly $k$-choosable graph is also $i$-strongly $(k+1)$-choosable and $(i-1)$-strongly $k$-choosable, which means that $\chi(G) =\chi^0_\ell(G) \le \chi^1_\ell(G) \le \cdots \le \chi^\infty_\ell(G) =\chi_\ell(G)$.
We will also show that $\chi^i_\ell(G) \le \chi^{i-1}_\ell(G)+\chi(G)-1$.
We say that $G$ is {\bf $i$-strongly chromatic-choosable}, if $\chi (G) =\chi^i_\ell(G)$.
Choosability with bounded number of used colors had been investigated in some papers; for example, see~\cite{Kang-2018, Kral-Sgall-2018}.
A graph $G$ is called {\bf claw-free}, if there is no triple of non-adjacent vertices having a common neighbour.
For a graph $G$, the {\bf line graph} $L(G)$ is a graph whose vertex set is $ E(G)$ and also two $e_1,e_2\in E(G)$ are adjacent in $L(G)$ if they have a common end in $G$. Note that line graphs are claw-free.
The {\bf total graph } $T(G)$ is a graph whose vertex set is $V(G)\cup E(G)$, and two $e_1,e_2\in E(G)$ are adjacent if they have a common end in $G$, two $v_1,v_2\in V(G)$ are adjacent if they are adjacent in $G$, and also $v\in E(G)$ and $e\in E(G)$ are adjacent if $e$ is incident with $v$ in $G$.
For a positive integer $k$, the {\bf $k$-th power} $G^k$ of a graph $G$ is a graph with the same vertex set and two vertices are adjacent if their distance in $G$ is at most $k$. 
For the special case $k=2$, the graph $G^2$ is called the {\bf square} of $G$.
We denote by $S(G)$ the {\bf subdivision graph} of $G$ which can be obtained from it by inserting a new vertex on each edge.
These graphs are bipartite and have girth at least $6$ when $G$ is a simple graph.
It is easy to check that total graphs are square of subdivision graphs that means $T(G)=S(G)^2$.
We denote by $K_n$ and $\overline{K_{n}}$ the complete graph of order $n$ and its complement.
Likewise, we denote by $K_{n_1*r_1, n_2*r_2}$ the {\bf complete multipartite graph} having $r_i$ parts of size $n_i$. 
When $r_2=0$, we only write $K_{n_1*r_1}$.
The {\bf join of two graphs} $G$ and $H$ is denoted by $G\vee H$ which is the graph obtained from them by joining every vertex of $G$ to every vertex of $H$.
For a positive integer $n$, we denote by $\mathbb{Z}_n$ the cyclic group of order $n$ with elements $1,\ldots, n$. 
For two positive integer $n$ and $k$ with $ k<n/2$, the {\bf generalized Petersen graph $P(n,k)$} refers to a graph 
with vertices $v_i$ and $u_i$ and edges $v_iv_{i+1}$, $v_iu_i$, and $u_iu_{i+k}$ where $i\in \mathbb{Z}_n$.
Note that $P(5,2)$ is the Petersen graph.

In 1997 Borodin, Kostochka, and Woodall~\cite{Borodin-Kostochka-1997} conjectured that total graphs are chromatic-choosable.
\begin{conj}{\rm (List Total Coloring Conjecture \cite{Borodin-Kostochka-1997})}
{Every graph $G$ satisfies $\chi(T(G))= \chi_\ell(T(G))$, where 
$T(G)=S(G)^2$.
}\end{conj}

Motivated by this conjecture, Kostochka and Woodall (2001)~\cite{Kostochka-Woodall-2001} proposed a stronger conjecture which says that the square of graphs are chromatic-choosable. They also confirmed this conjecture for many small graphs.
\begin{conj}{\rm (List Square Coloring Conjecture \cite{Kostochka-Woodall-2001})}\label{conj:Kostochka-Woodall-2001}
{Every graph $G$ satisfies $\chi(G^2)= \chi_\ell(G^2)$.
}\end{conj}

After a long time, Kim and Park (2015)~\cite{Kim-Park-2015, Kim-Park-2015-Bipartite} constructed some families of non-bipartite and bipartite counterexamples to this conjecture which the square of them are complete multipartite graphs and the smallest one contains $15$ vertices, see \cite[Figure 3]{Kim-Park-2015}.

\begin{thm}{\rm (\cite{Kim-Park-2015-Bipartite})}\label{conj:}
{There exists an infinite family of bipartite graphs $G$ satisfying $\chi(G^2) \neq \chi_\ell(G^2) $.
}\end{thm}

They also posed the following problems in their paper and partially answered the second problem by giving the upper bound of $6$ on $k$. Note that total graphs are square of bipartite graphs having degree $2$ in one side.
In this paper, we introduce cubic bipartite planar counterexamples to the List Square Coloring Conjecture which consequently shows that $k$ must be at most $2$ (if there would exist); see Figures~\ref{fig:CubicBipartitePlanar-24} and~\ref{fig:Smaller-planar-or-bipartite}. These examples also disprove a planar version of Conjecture~\ref{conj:Kostochka-Woodall-2001} in~\cite[Conjecture 6.4]{Havet-Heuvel-McDiarmid-Reed-2017} proposed by Havet, Heuvel, McDiarmid, and Reed (2017).
\begin{alphprob}{\rm (\cite{Kim-Park-2015-Bipartite})}\label{prob:Kim-Park-2015-Bipartite}
{If G is a bipartite graph such that every vertex of one partite set has degree at most $2$, then is it true that $\chi(G^2) = \chi_\ell(G^2)$?
}\end{alphprob}
\begin{alphprob}{\rm (\cite{Kim-Park-2015-Bipartite})}\label{prob:k:Kim-Park-2015-Bipartite}
{If the answer to Problem~\ref{prob:Kim-Park-2015-Bipartite} is yes, then what is the largest $k$ such that $G^2$ is chromatic-choosable for every bipartite graph G with a partite set in which each vertex has degree at most $k$?
}\end{alphprob}

We feel that the square of bipartite graphs must be chromatic-choosable provided that their the chromatic numbers are large enough compared to their maximum degree. In particular, we propose the following conjecture which still implies the List Total Coloring Conjecture if $c=1$. More precisely, for total graphs we have $\chi(T(G)) \ge \frac{1}{2}\Delta(T(G)) +1$. Note that the integer number $c$ cannot be reduced by $1/2$, because the bipartite graphs described in Theorem~\ref{thm:bipartite:planar} satisfies $ \chi(G^2)=4$ and $\Delta(G^2)=7$.

\begin{conj}{\rm (Modified Version of Conjecture~\ref{conj:Kostochka-Woodall-2001})}
{If $G$ is a bipartite graph, then $\chi(G^2)= \chi_\ell(G^2)$, provided that $\chi(G^2) \ge \frac{1}{2}\Delta(G^2)+c$ for a fixed positive integer $c$ with $c\ge 1$.
}\end{conj}

In 1997 Gravier and Maffray~\cite{Gravier-Maffray-2001} conjectured that claw-free graphs are chromatic-choosable which is a stronger version of the following conjecture due to
 Vizing, Gupa, Albertson and Collins, and Bollob{\' a}s and
Harris, see~\cite{Jensen-Toft-1995}. For bipartite graphs, Conjecture~\ref{intro:conj:List Coloring Conjecture} is confirmed in~\cite{Galvin-1995} completely.

\begin{conj}{\rm (List Coloring Conjecture)}\label{intro:conj:List Coloring Conjecture}
{Every graph $G$ satisfies $\chi(L(G))= \chi_\ell(L(G))$.
}\end{conj}

 Kim, Kwon, and Park (2015) and
Dai, Wang, Yang, and Yu (2018)~\cite{Dai-Wang-Yang-Yu-2018} posed the following problems in their paper about chromatic-choosability of claw-free graphs and line graphs. In this paper, we answer their problems negatively, for $k=2$, by giving several families of line graphs (see Theorems~\ref{thm:claw-free:cubic:prism},~\ref{thm:line:generalized:P:bipartite}, and~\ref{thm:line:generalized:P:planar}). In particular, we show that there are line (claw-free) graphs $G$ such that $G^2$ is not chromatic-choosable and $\chi(G^2) - \Delta(G^2)/2$ is sufficiently large (see Corollary~\ref{cor:square:graph:maximum-degree-chromatic}).
\begin{alphprob}{\rm (\cite{Kim- Kwon-Park-2015})}\label{prob:Kim- Kwon-Park-2015}
{Is $G^k$ chromatic-choosable for every integer $k\ge 2$ if $G$ is claw-free?
}\end{alphprob}

\begin{alphprob}{\rm (\cite{Dai-Wang-Yang-Yu-2018})}\label{prob:Dai-Wang-Yang-Yu-2018}
{Is $G^2$ chromatic-choosable for every line graph $G$?
}\end{alphprob}

In 2015 Noel, Reed, and Wu~\cite{Noel-Reed-Wu-2015} showed that graphs with small order are $k$-choosable by proving the following theorem which was originally conjectured by Ohba (2002)~\cite{Ohba-2002}. 
This result is also extended to graphs of order $2k+2$ in \cite{Zhu-Zhu-2022, Zhu-Zhu-arXiv} by proving that
 the two graphs $\overline{K_{4}} \vee (2K_{k-1})$ and $K_{1*\frac{k-2}{2},3*\frac{k+2}{2}}$ are all 
edge-minimal $k$-chromatic non-$k$-choosable graphs of order $2k+2$ provided that $k$ is even and $k\ge 4$. As a consequence, all $k$-chromatic non-$k$-choosable graphs $G$ of order $2k+2$ satisfy $ \Delta(G) \ge 2\chi(G)-2$. 
\begin{thm}{\rm (\cite{Noel-Reed-Wu-2015})}\label{thm:Noel-Reed-Wu-2015}
{If $G$ is a $k$-chromatic graph satisfying $|V(G)| \le 2k+1$, then $G$ is $k$-choosable.
}\end{thm}

Motivated by a computer search, we conjecture that the two graphs $\overline{K_{5}} \vee (2K_{k-1})$ and $K_{1*\frac{k-3}{2},3*\frac{k+3}{2}}$ are all edge-minimal $k$-chromatic non-$k$-choosable graphs of order $2k+3$ provided that $k$ is odd and $k\ge 5$. Moreover, by applying our computer program, we succeeded to characterize all edge-minimal $3$-chromatic non-$3$-choosable (resp $4$-chromatic non-$4$-choosable) graphs of order at most $9$ (reps. $11$) and observed that all of them have maximum degree at least $5$ and size at least $19$ (resp. $6$ and $30$); see Figures~\ref{9all} and~\ref{fig:at-most-11-all}. This settles a problem posed by Nelsen \cite[Question 3.19]{Nelsen-2019} about determining smallest non-$3$-choosable graphs.
\begin{thm}
{There exists a $3$-chromatic non-$3$-choosable graph of size $19$ (resp. maximum degree $5$).
}\end{thm}

In this paper, we put forward the following conjecture to strengthen Theorem~\ref{thm:Noel-Reed-Wu-2015} further in terms of maximum degree. To approach the best possible upper bound on $|V(G)|$, we will construct in Corollary~\ref{cor:2k+9:maximum-degree}, a $k$-colorable non-$k$-choosable graph $G$ satisfying $\Delta(G)= 2k-3$ and $|V(G)| = 2k+9$ provided that $k$ is divisible by $4$ and greater than $16$.
\begin{conj}\label{intro:conj:small-graphs}
{If $G$ is a $k$-chromatic graph satisfying $\Delta(G)\le 2k-3$
and $|V(G)| \le 2k+4$, then $G$ is $k$-choosable.
}\end{conj}

In Section~\ref{sec:last}, we will investigate non-choosable graphs with bounded maximum degree in bipartite or planar graphs. Consequently, we improve several graph constructions due to Erd\H os, Rubin, and Taylor~(1980)~\cite{Erdos-Rubin-Taylor-1980}, 
Bessy, Havet, and Palaysi (2002)~\cite{Bessy-Havet-Palaysi-2002}, Voigt (1993)~\cite{Voigt-1995},
Mirzakhani (1996)~\cite{Mirzakhani-1996}, and Glebov, Kostochka, and Tashkinov~(2005)~\cite{Glebov-Kostochka-Tashkinov-2005}
in terms of order or maximum degree.
In particular, we prove the following theorem for bipartite graphs.
\begin{thm}
{For every integer $k$ with $k\ge 3$, there exists a bipartite non-$k$-choosable graph of size $\sum^{k-1}_{i=0}\binom{k}{i}\binom{k+i}{k}$ (resp. $\sum^{k-1}_{i=0}\binom{k}{i}\binom{k+i}{k}+k^2-1$) having maximum degree $2^k-1$ (reps. $2^k-2$). 
}\end{thm}
\section{Solution to Problem~\ref{prob:k:Kim-Park-2015-Bipartite}: Cubic bipartite graphs and prisms}
The following theorem completely solves Problem~\ref{prob:k:Kim-Park-2015-Bipartite} 
 by showing that if Problem~\ref{prob:Kim-Park-2015-Bipartite} would be true, then $k$ must be precisely $2$ (recall that total graphs $T(G)$ are square of subdivision graphs $S(G)$ with girth at least $6$ when $G$ is a simple graph).
\begin{thm}\label{thm:bipartite}
{There exists an infinite family of bipartite cubic graphs $G$ with girth $6$ whose squares are not $1$-strongly chromatic-choosable.
}\end{thm}
\begin{proof}
{Let $n$ be a positive integer. Let $G$ be a cubic bipartite graph with vertices $x_i$ and $y_i$,
and edges $x_{i}y_{i-2}$, $x_iy_i$, $x_{i}y_{i+1}$ where $i\in \mathbb{Z}_n=\{1,\ldots, n\}$.
It is easy to check that $G$ has no cycle with size $4$.
For the case $n=8$, the graph $G$ is illustrated in Figure~\ref{CubicBipartiteGirth6} 
by Hamiltonian cycle $y_{1}x_1y_{2}x_2\cdots y_nx_n$ (anti-clockwise order).
\begin{figure}[h]
 \centering
 \includegraphics[scale = 1.48]{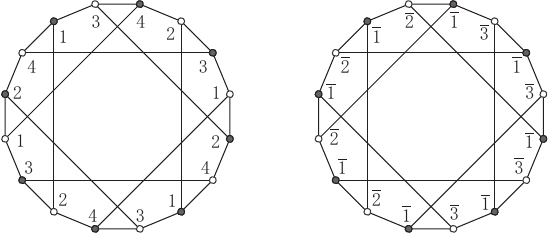}
 \caption{The square of the graph $G$ is $4$-colorable (left) but not $4$-choosable (right)} 
\label{CubicBipartiteGirth6}
\end{figure}
We claim that $\chi^1_\ell(G^2) > \chi(G^2)$ provided that $n$ is divisible by $4$.
By the definition, $G^2$ contains the edges of $G$ along with the new edges
$x_ix_{i+1}$, $x_ix_{i+ 2}$, $x_ix_{i+ 3}$, and also
$y_iy_{i+ 1}$, $y_iy_{i+ 2}$, $y_iy_{i+ 3}$ where $i\in \mathbb{Z}_n$.
Since $G^2$ has the clique number $4$, we must have $4 \le \chi(G^2)$ (any four consecutive vertices of $x_i$ or $y_i$ is a clique).
Assume that $n$ is divisible by $4$.
Then we can find a $4$-coloring for $G^2$ by coloring every vertex $x_{i}$ and $y_{i-1}$ by the same color $r$, where $i\stackrel{4}{\equiv}r\in \mathbb{Z}_4$.
Now, we are going to show that $\chi^1_\ell (G^2) > 4$ whenever $n\ge 8$.
We assign the list $\bar{1}$ on all vertices $x_1,\ldots, x_n$, 
assign the list $\bar{2}$ on the vertices $y_1,\ldots, y_{n-4}$, and assign
the list $\bar{3}$ on all vertices $y_{n-3},\ldots, y_{n}$, where $\bar{i}=\{1,\ldots, 5\}\setminus \{i\}$. 
Suppose, to the contrary, that $G^2$ admits such a list coloring.
By the property of $G^2$, every vertex $x_i$ must be colored by $c_r$, where $i\stackrel{4}{\equiv}r\in \mathbb{Z}_4$ and $ \{c_1,c_2,c_3,c_4\}= \{2,3,4,5\}$.
Therefore, every vertex $y_{i-1}$ must be colored by the color $1$ or $c_r$, where $i\stackrel{4}{\equiv}r\in \mathbb{Z}_4$.
Assume that $c_{r_2} =2$ and $c_{r_3} =3$.
According to the list property, if $1\le i-1 \le n-4$ and $i\stackrel{4}{\equiv}r_2$, then every vertex $y_{i-1}$ must be colored by the color $1$. Likewise, if $n-4< i-1 \le n$ and $i\stackrel{4}{\equiv}r_3$, then $y_{i-1}$ must be colored by the color $1$.
This is a contradiction and the proof is completed.
}\end{proof}
In the following theorem, we introduce another family of bipartite cubic graphs which are planar and consequently they must have some cycles with size $4$. 
\begin{thm}\label{thm:bipartite:planar}
{There exists an infinite family of planar bipartite cubic graphs $G$ whose squares are not $1$-strongly chromatic-choosable.
}\end{thm}
\begin{proof}
{Let $n$ be a positive integer. Let $G$ be a cubic bipartite graph with vertices $x_i$ and $y_i$,
and edges $x_{i}x_{i+1}$, $x_iy_i$, $y_{i}y_{i+1}$ where $i\in \mathbb{Z}_n=\{1,\ldots, n\}$.
We claim that $\chi^1_\ell(G^2) > \chi(G^2)$ provided that $n$ is divisible by $4$.
By the definition, $G^2$ contains the edges of $G$ along with the new edges
$x_ix_{i+1}$, $x_ix_{i+ 2}$, and 
$y_iy_{i+ 1}$, $y_iy_{i+ 2}$, where $i\in \mathbb{Z}_n$.
Since $G^2$ has the clique number $4$, we must have $4 \le \chi(G^2)$ (any four vertices $x_i, y_i, x_{i+1}, y_{i+1}$).
Assume that $n$ is divisible by $4$.
Then we can find a $4$-coloring for $G^2$ by coloring 
any two vertices $x_{i}$ and $y_{i-2}$ by the same color $c_i \in \mathbb{Z}_4$ such that $c_i \stackrel{4}{\equiv}i\in \mathbb{Z}_n$.
Now, we are going to show that $\chi^1_\ell(G^2) > 4$ provided that $n\ge 12$.
We assign the list $\bar{4}$ on all vertices $x_i$ and $y_i$ with $8 \le i \le n$, 
assign the list $\bar{3}$ on all six vertices $x_i$ and $y_i$ with $i\in \{3,4,5\}$, 
assign the list $\bar{2}$ on all four vertices $x_1,x_2, y_6, x_7$, 
assign the list $\bar{1}$ on the remaining four vertices, $y_1,y_2, x_6, y_7$, 
where $\bar{i}=\{1,\ldots, 5\}\setminus \{i\}$. 

\begin{figure}[h]
 \centering
 \includegraphics[scale = 1.55]{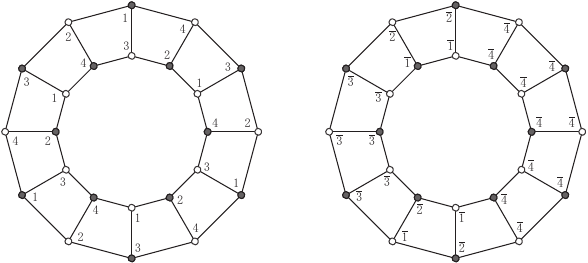}
 \caption{The square of the planar bipartite graph $P(4k, 1)$ is $4$-colorable (left) but not $4$-choosable (right)} 
\label{fig:CubicBipartitePlanar-24}
\end{figure}

Suppose, to the contrary, that $G^2$ admits such a list coloring $c:V(G)\rightarrow \mathbb{Z}_5$.
We may assume that $n=12$, because the color of $x_i$ (resp. $y_i$) must be repeated on $x_{i+4}$ (resp. $y_{i+4}$), provided that $8 \le i \le n-4$.
In addition, $c(x_{8}) = c(y_{10})= c(x_{12})$, $c(y_{8}) = c(x_{10})= c(y_{12})$, and 
$c(x_{9}) = c(y_{11})$, $c(y_{9}) = c(x_{11})$. 
Let $H$ be the bipartite induced subgraph of $G^2$ consisting of all vertices colored with the color $1$ or $2$.
According to the list property, it is not difficult to check that there is no integer $i$ such that
$\{c(x_i), c(y_i)\}= \{1,2\}$. 
So, by the symmetry property, we can assume that $x_3, y_4, y_5\in V(H)$.
Let $v\in \{x_1,x_2,y_1,y_2\}\cap V(H)$ and let $u\in \{x_6,y_6,x_7,y_7\}\cap V(H)$.
Since $H$ has no triangle, we must have $v\neq y_2$ and $u\neq y_6$.
We claim also that $v$ and $u$ do not have lists $\bar{2}$ and $\bar{1}$, respectively.
Otherwise, five vertices $v, x_3,y_4,y_5,u$ form a path in $H$ whose lists are $\bar{2}$, $\bar{3}$, $\bar{3}$, $\bar{3}$, and $\bar{1}$, respectively, which is impossible.
In addition, if $\{v,u\}=\{x_1,y_7\}$ or $\{v,u\}=\{y_1,x_7\}$, then $c(u)=c(v)\in \{c(x_9),c(y_9)\}$ which is impossible, because those two vertices have different lists $\bar{1}$ and $\bar{2}$.
On the other hand, if $\{v,u\}=\{x_1,x_7\}$ or $\{v,u\}=\{y_1,y_7\}$, then $c(u)\neq c(v)$ (and $\{c(u),c(v)\}= \{c(x_9),c(y_9)\}$ which is again impossible, because those two vertices the same list either $\bar{1}$ or $\bar{2}$.
Therefore, we have only two cases $\{v,u\}=\{y_1, x_6\}$ or $\{v,u\}=\{x_2, y_7\}$.
In both cases, it is not hard to verify that the color $3$ must appear at most $4$ times. 
On the other hand, every color $i$ must appear at most $n_i$ times, where $n_1=n_2=5$, $n_3=n_4=4$, $n_5=6$.
Therefore, every color $i$ must appear exactly $n_i$ times, because $G$ contains $24$ vertices.
This can imply that $c(x_1)=c(y_3)=c(x_5)=c(y_7)=4$ and hence $c(y_2)=c(x_4)=c(y_6)=c(x_8)=c(y_{10})=c(y_{12})=5$.
Thus if $\{v,u\}=\{y_1, x_6\}$, then one can conclude that $c(x_7) = 3$, and 
similarly if $\{v,u\}=\{x_2, y_7\}$, then $c(y_1) = 3$.
In both cases, the color of $y_1$ and $x_7$ are different and not equal to $4$ (they must consequently have the same color of $y_{10}$), which can derive contradiction.
}\end{proof}
\begin{figure}[h]
 \centering
 \includegraphics[scale = 1.3]{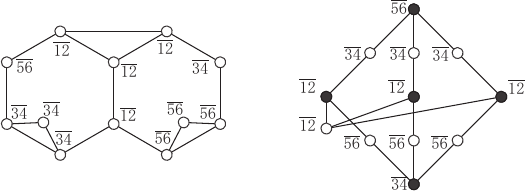}
 \caption{Smaller planar or biparite graphs whose square graphs are $4$-colorable but not $4$-choosable} 
\label{fig:Smaller-planar-or-bipartite}
\end{figure}
%
%
%
\section{Solution to Problems~\ref{prob:Kim- Kwon-Park-2015} and~\ref{prob:Dai-Wang-Yang-Yu-2018}}
\subsection{Line graph of subdivision of prisms and complete graphs}
\label{sec:subdivision}
The List Total Coloring Conjecture \cite{Borodin-Kostochka-1997} says that the square of subdivision graphs are chromatic-choosable. In this subsection, we shall show that there are some subdivision graphs for which the square graph of their line graphs are not chromatic-choosable.
\begin{thm}\label{thm:claw-free:cubic:prism}
{There exists an infinite family of claw-free planar cubic graphs $G$ whose squares are not $1$-strongly chromatic-choosable.
In particular, $G=L(S(P(4k,1)))$.
}\end{thm}
\begin{proof}
{Let $n$ be a positive integer divisible by $4$. Let $G$ be the claw-free cubic graph with vertices $x_{i, j}$, $y_{i, j}$, $z_{i, j}$ for which $j\in \{1,2\}$
and edges $x_{i,1}x_{i,2}$, $y_{i,1}y_{i,2}$, $z_{i,1}z_{i,2}$, and $z_{i,1}x_{i,1}$, $z_{i,1}x_{i,2}$, $z_{i,2}y_{i,1}$, $z_{i,2}y_{i,2}$, and also
$x_{i,2} x_{i+1,1}$, $y_{i,2} y_{i+1,1}$, where $i\in \mathbb{Z}_n=\{1,\ldots, n\}$.
Obviously, $G$ is a planar graph of order $6n$ and in particular, $G=L(S(P(n,1)))$.
We claim that $\chi^1_\ell(G^2) > \chi(G^2)$.
Since $G^2$ has the clique number $4$, we must have $4 \le \chi(G^2)$ (for example, any four vertices $x_{i,1}, x_{i,2}, z_{i,1}, z_{i,2}$).
Then we can find a $4$-coloring for $G^2$ by coloring 
any six vertices $z_{i,1}$, $y_{i+1,1}$, $x_{i+1,2}$, $z_{i+2,2}$, $x_{i+3,1}$, $y_{i+3,2}$ by the same color $c_i \in \mathbb{Z}_4$ such that $c_i \stackrel{4}{\equiv}i\in \mathbb{Z}_n$.
Now, we are going to show that $\chi^1_\ell(G^2) > 4$. Let
$C$ be the $8$-cycle $x_{n-1, 2} x_{n, 1} z_{n, 1}z_{n, 2} y_{n, 1} y_{n-1, 2}z_{n-1, 2}z_{n-1, 1}$ of $G$.
We assign the list $\bar{1}$ to all vertices except the vertices in $V(C)$; see Figure~\ref{SubdivisionPrism} (right part).
We assign the list $\bar{2}$ to the four vertices in $V_2=\{x_{n-1, 2}, x_{n, 1}, y_{n-1, 2}, y_{n, 1}\}$, 
and assign the list $\bar{3}$ to the four vertices in $V_3=\{z_{n-1, 1}, z_{n-1, 2}, z_{n, 1},z_{n, 2}\}$,
where $\bar{i}=\{1,\ldots, 5\}\setminus \{i\}$;
see Figure~\ref{SubdivisionPrism} for the smallest case $n=4$.

\begin{figure}[h]
 \centering
 \includegraphics[scale = 1.7]{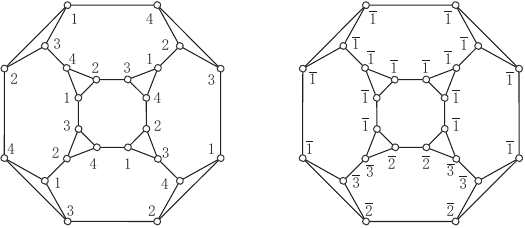}
 \caption{The square of the planar line graph $L(S(P(4,1)))$ is $4$-colorable (left) but not $4$-choosable (right)} 
\label{SubdivisionPrism}
\end{figure}

Suppose, to the contrary, that $G^2$ admits such a list coloring $c:V(G)\rightarrow \mathbb{Z}_5$.
Let $c_1,\ldots, c_4$ be the colors of the vertices $z_{1, 1}$, $x_{1, 1}$, $z_{1, 2}$, $x_{1, 2}$, respectively.
Since these vertices form a clique of size four, these colors must be distinct and so $\{c_1,\ldots, c_4\}=\{2,\ldots, 5\}$.
It is not to difficult to show that all vertices in 
$\{z_{i,1}, y_{i+1,1}, x_{i+1,2}, z_{i+2,2}, x_{i+3,1}, y_{i+3,2}\}\setminus V(C)$ receive the same color $c_i$. 
In fact, we can determine uniquely the colors of the vertices $\{x_{i, 1}, x_{i, 2}, y_{i, 1}, y_{i, 2}, z_{i, 1},z_{i, 2}\}\setminus V(C)$ by induction on $i$.
Therefore, one can also conclude that the color of every vertex $v\in C_i$ lies in the set $\{c_i, 1\}$, where 
 $C_1=\{z_{n-1, 2}, x_{n, 1}\}$, 
$C_2=\{x_{n-1, 2}, z_{n, 2}\}$, 
$C_3=\{z_{n-1, 1}, y_{n, 1}\}$, and
$C_4=\{y_{n-1, 2}, z_{n, 1}\}$.
Assume that $c_{a_2}=2$ and $c_{a_3}=3$. 
Then the color $1$ must appear on the unique vertex in $V_i\cap C_{a_i}$ with the given list $\overline{i}$, where $i\in \{2,3\}$.
This distribution of the color $1$ is possible only if $a_1=a_2$ according to the graph structure of $C^2$. 
This is contradiction, as desired.
}\end{proof}
The smallest graph introduced in Theorem~\ref{thm:claw-free:cubic:prism} contains $24$ vertices and its smallest cycles, except triangles, have size $8$. In the following theorem, we introduce a smaller graph having only $12$ vertices but its smallest cycles, except triangles, have size $6$. 
\begin{thm}\label{thm:claw-free:cubic}
{There exists a claw-free planar cubic graph $G$ of order $12$ whose square is not $1$-strongly chromatic-choosable.
In particular, $G=L(S(K_{4})))$.
}\end{thm}
\begin{proof}
{Let $G$ be the graph shown in Figure~\ref{fig:12-cubic-claw-free}.
Since $G^2$ has the clique number $4$, we must have 
$\chi(G^2)\ge 4$.
To show the equality, it is enough to consider the $4$-coloring shown in Figure~\ref{fig:12-cubic-claw-free}.
We claim that $\chi^1_\ell(G^2) > \chi(G^2)$ which can complete the proof.
To show this, first consider the list assignment for vertices as Figure~\ref{fig:12-cubic-claw-free}, where
 $\overline{i}=\{1,2,3,4,5\}\setminus \{i\}$.
\begin{figure}[h]
 \centering
 \includegraphics[scale = 1.7]{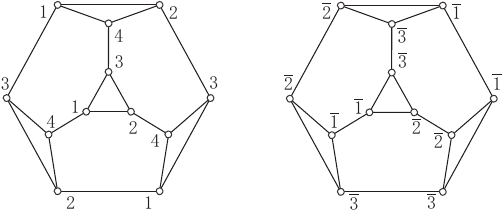}
 \caption{The square of the line graph $L(S(K_{4}))$ is $4$-colorable (left) but not $4$-choosable (right)} 
\label{fig:12-cubic-claw-free}
\end{figure}

Suppose, to the contrary, that $G^2$ admits such a list coloring.
It is easy to check that $G^2$ contains exactly four maximum independent sets $V_1,\ldots, V_4$ of size $3$ which are shown in Figure~\ref{fig:12-cubic-claw-free} (left part).
According to the list assignment, every $V_i$ contains all three lists $\overline{1}, \overline{2}$, and $\overline{3}$.
Thus each of the colors $1$, $2$, $3$ cannot appear three times.
 Therefore, two colors $4$ and $5$ must appear exactly three times and each of the 
colors $1$, $2$, $3$ must appear exactly two times.
After removing the vertices colored by $4$ and $5$ in the graph $G^2$,
the remaining graph is the union of two $4$-cycles having an edge in common. 
In particular, if we consider two mentioned cycles $w_0w_1v_0v_1$ and $w'_0w'_1v_0v_1$, then
 lists on the vertices of them must be 
$ \overline {c_1}, \overline {c_2}, \overline {c_3}, \overline {c_3}$
 and 
$ \overline {c_2}, \overline {c_1}, \overline {c_3}, \overline {c_3}$, respectively, where $c_1$, $c_2$, and $ c_3$ are a permutation of $1$, $2$, $3$. But this subgraph does not admit such a list coloring; for example, 
see the middle graph in Figure~\ref{smallest-non-chromatic-choosable-graph}. This is a contradiction, as desired.
Hence the proof is completed.
}\end{proof} 
We shall below generalize the graph of Theorem~\ref{thm:claw-free:cubic} to $5$-regular graphs using a little extra effort.
\begin{thm}\label{thm:5-regular}
{There exists a claw-free $5$-regular graph $G$ of order $30$ 
whose square is not $1$-strongly chromatic-choosable.
In particular, $G=L(S(K_{6}))$.
}\end{thm}
\begin{proof}
{Let $G$ be a graph with vertex set 
$V(G)=\{v_{i,j}: i,j\in \mathbb{Z}_6, i\ne j\}$ and edge set 
$E(G)=\{v_{i,j}v_{j,i}: i,j\in \mathbb{Z}_6, i\ne j\}\cup \{v_{i,j}v_{i,j'}: i,j,j'\in \mathbb{Z}_6, j'\ne i\ne j\neq j'\}$, where $\mathbb{Z}_6=\{1,\ldots, 6\}$.
Note that $G=L(B(K_{6}))$.
Since $G^2$ has the clique number $6$, we must have $6\le \chi(G^2)$. 
 To prove the equality, we can color every vertex $v_{i,j}$ with the color $j$; for example, see Figure~\ref{FiveRegularGraph}.
Now, we shall prove that $\chi^1_\ell(G^2) \neq 6$.
Let $L:\{1,\ldots, 6\}\times \{1,\ldots, 6\} \rightarrow \{0,\ldots, 5\}$ be a Latin square obtained from a one-factorization of the complete graph $K_6$.
More precisely, every edge $x_ix_j$ in the complete graph must have the color $L(i,j)\in \{1,\ldots, 5\}$ (consequently, $L(i,i) =0$). 
We assign the color list $\overline {L(i,j)}= \{0,\ldots, 6\}\setminus \{L(i,j)\}$ to each vertex $v_{i,j}$.
For example, a list assignment is shown in Figure~\ref{FiveRegularGraph}.

 Suppose, to the contrary, that $G^2$ admits such a list coloring. 
For each $i$ with $1\le i\le 6$, we set $V_i=\{v_{i,j}: j\in \mathbb{Z}_6, i\ne j\}$.
Let $C$ be the multigraph with vertex set $V(C) = \{0, \ldots, 6\}$ for which edge set
consists of edges $ij$ such that both colors $i$ and $j$ do not appear on all vertices of a fixed partition $V_t$.
Clearly, $C$ has size $6$.
We first claim that $C$ is a tree. 
Suppose, to the contrary, that $C$ has a cycle $c_1,\ldots, c_n$.
Let $V_{p_i}$ be the partition corresponding to the edge $c_{i-1}c_{i}$, where $i\in \{1,\ldots, n\}$; we here consider $c_0$ for $c_n$.
Since the vertex $v_{p_n, p_i}$ is adjacent to all vertices of $V_{p_i}$, the color of it must be either $c_{i-1}$ or $c_{i}$.
In particular, the color of $v_{p_n, p_1}$ must be either $c_n$ or $c_1$.
Hence it must be $c_1$, because of $v_{p_n, p_1}\in V_{p_n}$.
 Thus one can inductively conclude that the color of 
$v_{p_n, p_i}$ must be $c_{i}$, since vertices $v_{p_n, p_i}$ and $v_{p_n, p_{i-1}}$ are adjacent.
This implies that the color of $v_{p_n, p_{n-1}}$ must be $c_{n-1}$ which is impossible, because of $v_{p_n, p_{n-1}}\in V_{p_n}$.
Now, we claim that every vertex $i\in V(C)$ satisfying $1\le i\le 5$ must have degree exactly two.
Otherwise, there is a color $i_0\in V(C)$ with $1\le i_0\le 5$ such that appears on all partitions except a fixed partition $V_t$.
Therefore, this color must appear on all vertices $v_{j, t}$ where $j\in Z_6$ and $j\neq t$
(in fact, every vertex $v_{j, t'}$ is adjacent to all vertices of $V_{t'}$ and so those have different colors).
On the other hand, these vertices have all lists $\bar{1}, \ldots,\bar{5}$, which is impossible.
Therefore, we can assume that $C$ is a path with vertices $c_{0},c_{1},\ldots, c_{6}$.

Let $V_{p_i}$ be the partition corresponding to the edge $c_{i-1}c_{i}$, 
where $i\in \mathbb{Z}_6$.
As the above-mentioned argument, one can inductively conclude that the color of $v_{p_i, p_j}$ must be $c_{j}$ whenever $i < j$, and $c_{j-1}$ whenever $j < i$ (the argument for a fixed integer $i$ is based on two inductions, the first one uses the integer $j$ starting from $i+1$ to $6$ and the other one uses the integer $j$ starting from $i-1$ to $0$). 
Therefore, the color $c_0$ must appear on all vertices $v_{p_i, p_1}$ with $2\le i \le 6$.
Since these vertices have all lists $\bar{1}, \ldots,\bar{5}$, one can conclude that $c_0\in \{0,6\}$.
Likewise,
 the color $c_6$ must appear on all vertices $v_{p_i, p_6}$ with $1\le i\le 5$, and hence $c_6\in \{0,6\}$.
Furthermore, the color $c_1$ must appear on all vertices 
$v_{p_i, p_2}$ with $3\le i \le 6$.
This implies that $c_1=L(p_1,p_2)$.
Likewise, we must have $c_5=L(p_6,p_5)$.
Therefore, one can similarly conclude that $c_2\in \{L(p_1,p_3), L(p_2,p_3)\}$ and
 $c_4\in \{ L(p_5,p_4), L(p_6,p_4)\}$.
Since the color $c_4$ (resp. $c_5$) is a perfect matching of $K_6$, it must be in $\{ L(p_1,p_3), L(p_2,p_3)\}$. This is a contradiction because $c_2$, $c_4$, and $c_5$ are three different colors.
Hence the proof is completed.
}\end{proof}
\begin{figure}[h]
 \centering
 \includegraphics[scale = 0.9]{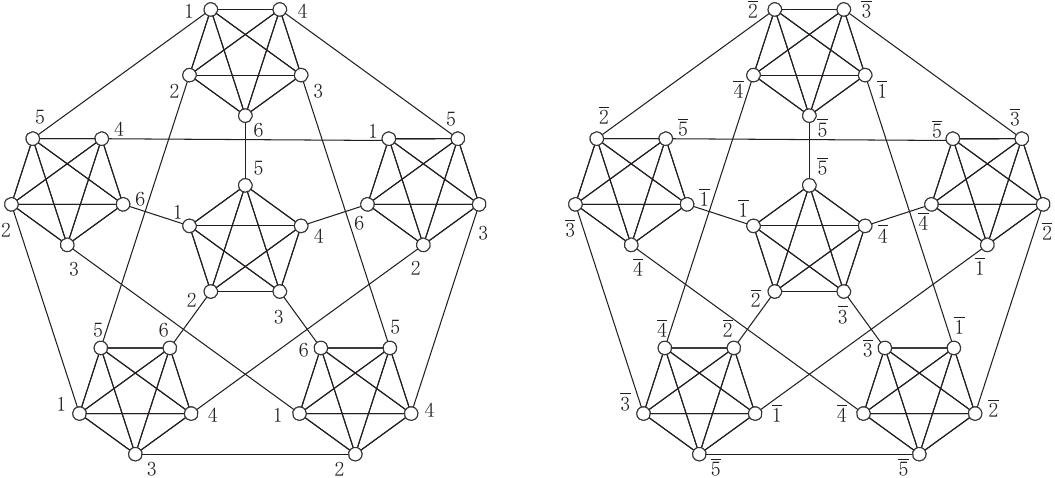}
 \caption{The square of the line graph $L(S(K_{6}))$ is $6$-colorable (left) but not $6$-choosable (right)} 
\label{FiveRegularGraph}
\end{figure}
\subsection{Line graph of generalized Petersen graphs}
Kim and Park (2015)~\cite{Kim-Park-2015} introduced a small graph of order $15$, with degrees $3$ and $4$, whose square graph is not chromatic-choosable; the right graph illustrated in Figure~\ref{fig:15-Petersen-KimPark}.
We observed that line graph of the Petersen graph is another small graph having the same square graph.
More precisely, their square graph are the multipartite graph whose parts have size three and so it must have the chromatic number and the list chromatic number $5$ and $7$, see \cite{Kierstead-2000} (note that there are some other strongly regular graphs whose line graphs are complete multipartite; for example, Clebsch graph).
\begin{figure}[h]
 \centering
 \includegraphics[scale = 1.52]{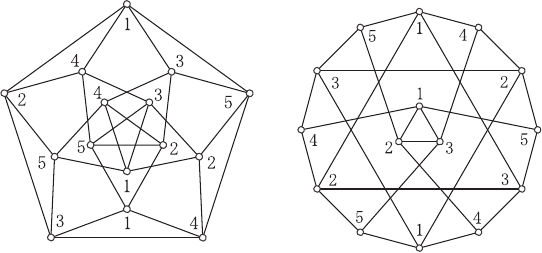}
 \caption{The line graph of the Petersen graph (left) and the Kim-Park graph (right) have the same square graph: complete multipartite graph with part size $3$ and chromatic number $5$.}
 \label{fig:15-Petersen-KimPark}
\end{figure}
In the following theorem, we are going to show that the square graph of the line graph of a bipartite cubic graph is not necessarily chromatic-choosable (using generalized Petersen graphs). We have already observed that the square graph of a bipartite cubic graph is not necessarily chromatic-choosable.
\def\x {1}
\def\y {2}
\def\z {3}
\begin{thm}\label{thm:line:generalized:P:bipartite}
{There exists an infinite family of claw-free $4$-regular graphs $G$ whose squares are not $1$-strongly chromatic-choosable.
In particular, $G$ is the line graph of a bipartite cubic graph.
}\end{thm}
\begin{proof}
{Let $P$ be a generalized Petersen graph $P(n, 3)$ 
with vertices $v_i$ and $u_i$ and 
edges $v_iv_{i+1}$, $v_iu_i$, and $u_iu_{i+3}$
where $i\in \mathbb{Z}_n=\{1,\ldots, n\}$ and $n$ is divisible by $5$. Note that $P$ is bipartite when $n$ is divisible by $10$.
Let $G$ be the line graph of $P$.
For notational simplicity, we use $x_i$, $y_i$ and $z_i$ in $V(G)$ corresponding to the edges 
$v_iv_{i+1}$, $v_iu_i$, and $u_iu_{i+3}$.
It was known that $\chi(G^2)=5$, see~\cite[Theorem 4.2]{Yang-Wu-2018}.
To see this it is enough to color any triple of vertices $x_i$, $y_{i+3}$, and $z_{i+4}$ by the same color $r$ where 
$i\stackrel{5}{\equiv}r\in \mathbb{Z}_5$.
Now, we claim that $\chi^1_{\ell}(G^2) > \chi(G^2)$. 
One method for proving it is to assign the same list $\bar{\x}$ on all $x_i$, and
assign the same list $\bar{\y}$ on all $y_i$, 
assign the list the same list $\bar{\z}$ on the remaining vertices, where $\bar{j}=\mathbb{Z}_6\setminus \{j\}$; for example, see Figure~\ref{20P10-2}. 
Unfortunately, this method only works whenever $n$ is not divisible by $4$ and the proof needs a little extra effort (we leave the proof for the reader). 
To introduce a simpler proof for our purpose, let us change the list assignments and assume that $n$ is at least $15$.
We assign the list $\bar{\x}$ on all $x_i$ with $n-8\le i\le n$,
assign the list $\bar{\y}$ on all $y_i$ with $n-8\le i\le n$,
and assign the list the list $\bar{\z}$ on the remaining vertices, where $\bar{j}=\mathbb{Z}_6\setminus \{j\}$.

Suppose, to the contrary, that $G^2$ admits such a list coloring $c:V(G)\rightarrow \mathbb{Z}_6$.
Let $i$ be an integer with $1\le i\le n-8$. 
We denote the colors of the vertices $x_i$, $x_{i+1}$, $y_i$, $y_{i+1}$, and $y_3$ by $a_0, a_1, b_0, b_1$, respectively.
Since these colors are different,
we can denote by $b_2$ the unique color in $\mathbb{Z}_6\setminus \{a_0, a_1, b_0, b_1,\z\}$.
If $i\le n-9$, then according to the list assignments, we must have $\{c(x_{i+2}),c(y_{i+2})\}=\{b_0,b_2\}$.
This implies that $\{c(x_{i+3}),c(y_{i+3})\}=\{b_1,a_0\}$ and so $\{c(z_{i+3}),c(z_{i})\}=\{a_1,b_2\}$.
Therefore, $c(x_{i+2})=b_0$ and $c(y_{i+2})\}=b_2$.
Let call the color of vertices $x_1$, $x_2$, $x_3$, $y_2$, and $y_3$ by $p_1,\ldots, p_5\in \mathbb{Z}_6\setminus \{3\}$.
By the argument mentioned above, one can conclude that for every $i$ with $1 \le i < n-7$, we must have
 $c(x_i)=p_j$ in which $i\stackrel{5}{\equiv}j\in \mathbb{Z}_5$, 
and $c(y_i)=p_j$ in which $i\stackrel{5}{\equiv}j-2\in \mathbb{Z}_5$.
Moreover, if $i=n-7$ then $\{c(x_{i}),c(y_{i})\}=\{p_j,p_{j'}\}$ where $i\stackrel{5}{\equiv}j\in \mathbb{Z}_5$ and $i\stackrel{5}{\equiv}j'-2\in \mathbb{Z}_5$.
Therefore, if $2 \le i \le n-10$, then $c(z_i)=p_j$ in which $i\stackrel{5}{\equiv}j-1\in \mathbb{Z}_5$.
Consequently, this equality holds for every $i\in \{1, n-9,n-8,n-7,n-2,n-1,n\}$ and then for all $i\in \{1,\ldots, n\}$.
These imply that 
for every $i$ with $n-6 \le i \le n$, we must have
 $c(x_i)\in \{p_j,3\}$ in which $i\stackrel{5}{\equiv}j\in \mathbb{Z}_5$, 
and $c(y_i')\in \{p_{j'},3\}$ in which $i'\stackrel{5}{\equiv}j'-2\in \mathbb{Z}_5$.
Thus if $p_j=1$, then $c(x_i)=3$, and if $p_{j'}=2$, then $c(y_{i'})=3$.
This is impossible because there are two adjacent vertices $x_i$ and $y_{i'}$ in $G^2$ with this properties.
}\end{proof}
\begin{figure}[h]
 \centering
 \includegraphics[scale = 1.5]{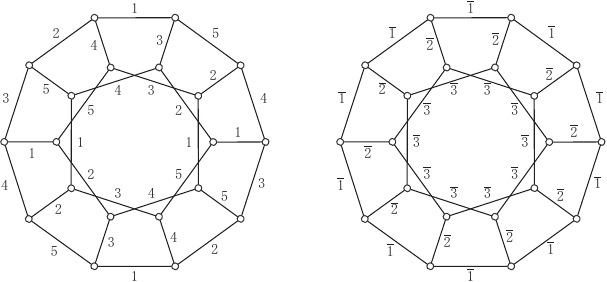}
 \caption{The square of the $4$-regular planar line graph $L(P(10,2))$ is $5$-colorable (left) but not $5$-choosable (right)} 
\label{20P10-2}
\end{figure}
In the following theorem, we are going to show that the square of the line graph of a planar cubic graph is not necessarily chromatic-choosable (using generalized Petersen graphs). We have already observed that the square of a planar cubic graph is not necessarily chromatic-choosable. This result answers a question by Cranston (2023) \cite[Page 9]{Cranston-2023} as well.
\begin{thm}\label{thm:line:generalized:P:planar}
{There exists an infinite family of planar claw-free $4$-regular graphs $G$ 
whose squares are not $1$-strongly chromatic-choosable. In particular, $G$ is the line graph of a planar cubic graph.
}\end{thm}
\begin{proof}
{Let $P$ be a generalized Petersen graph $P(n, 2)$ with vertices $v_i$ and $u_i$ and edges $v_iv_{i+1}$, $v_iu_i$, $u_iu_{i+2}$, 
where $i\in \mathbb{Z}_n=\{1,\ldots, n\}$ and $n$ is divisible by $5$. 
Let $G$ be the line graph of $P$. Note that $P$ and $G$ are planar when $n$ is divisible by $10$.
For notational simplicity, we use $x_i$, $y_i$ and $z_i$ in $V(G)$ corresponding to the edges 
$v_iv_{i+1}$, $v_iu_i$, and $u_iu_{i+2}$.
It was known that $\chi(G^2)=5$, see~\cite[Theorem 3.1]{Yang-Wu-2018}.
To see this, it is enough to color any triple of vertices $x_i$, $y_{i+3}$, and $z_{i+2}$ by the same color $r$ where $i\stackrel{5}{\equiv}r\in \mathbb{Z}_5$.

Now, we are going to show that $\chi^1_\ell(G^2) > \chi(G^2)$. 
We assign the list $\bar{1}$ on all $x_i$,
assign the list $\bar{2}$ on all $y_i$,
and assign the list the list $\bar{3}$ on the remaining vertices, where $\bar{j}=\mathbb{Z}_6\setminus \{j\}$.
Suppose, to the contrary, that $G^2$ admits such a list coloring $c:V(G)\rightarrow \mathbb{Z}_6$.
First assume that colors of any two vertices $y_{i}$ and $z_{i-1}$ are the same.
Let contract any pair of such vertices and call the resulting induced subgraph with these new vertices by $H$. If we consider the coloring for $H$ obtained from $G$, this coloring uses at most four colors of
$\mathbb{Z}_6\setminus \{2,3\}$. This is a contradiction, since $H$ contains some cliques with size $5$.

Now, we may assume that there exists an integer $i$ such that all colors of the vertices 
$x_{i},y_{i},z_{i}
,x_{i+1},y_{i+1},
 y_{i+2}$ are different.
We denote them by $a_1,\ldots, a_6$, respectively.
Hence $c(x_{i+2})=a_2$ and $c(x_{i-1})=a_6$.
Since $c(z_{i+2})\in \{ a_1,a_5\}$ and $c(z_{i-2})\in \{ a_4,a_5\}$, 
according to the symmetry property of $G$, 
we may assume that $c(z_{i+2})= a_1$ (note that $z_{i+2}$ and $z_{i-2}$ are adjacent).
In addition, it is easy to check that 
$c(x_{i+3})\in \{a_3,a_5\}$, 
$c(y_{i+3})\in \{a_1,a_3\}$, 
$c(y_{i+4})\in \{a_4,a_5\}$, 
and
$c(z_{i+1})\in \{a_3,a_6\}$.
If $c(z_{i+1})=a_3$, then we must have
$c(x_{i+3})=a_5$,
$c(y_{i+3})=a_1$, and
$c(z_{i+2})=a_4$.
These imply that $1,2\not \in \{a_1,\ldots, a_6\}\setminus \{a_3\}$ which is a contradiction.

Therefore, $c(z_{i+1})=a_6$.
On the other hand, according to the list property, $1\in \{a_3,a_5\}$.
We shall consider the following two cases.
Case (1): $a_5=1$.
In this case, we must have
 $c(x_{i+3})=a_3$,
 $c(y_{i+3})=a_1$, 
and 
$c(z_{i+3})=a_4$ which is again a contradiction.
These imply that $3\not \in \{a_1,\ldots, a_6\}\setminus \{a_5\}$ which is a contradiction.
Case (2): $a_3=1$.
 In this case, 
$c(x_{i+3})=a_5$
If $c(y_{i+3})=a_3=1$, then one can derive a contradiction similarly to Case (1).
Thus
 $c(y_{i+3})=a_1$, 
$c(x_{i+3})=a_5$, and
$c(y_{i+4})=a_4$.
These imply that $2\not \in \{a_1,\ldots, a_6\}\setminus \{a_3\}$ which is again a contradiction.
Hence the proof is completed.
}\end{proof}
Mirzakhani (1996)~\cite{Mirzakhani-1996} constructed a $3$-colorable planar graph $G$ satisfying $\chi^1_\ell(G) \ge \chi(G) +2$. We would like to know whether there are such square graphs obtained from planar graphs.
\begin{prob}\label{prob:2constant}
{Is there a planar graph $G$ satisfying $\chi^1_\ell(G^2) \ge \chi(G^2) +2$?
}\end{prob}
\subsection{Line graph of bipartite planar graphs with large maximum degree}

In this subsection, we are going provide some counterexamples to Conjecture~\ref{conj:Kostochka-Woodall-2001} (and Problem~\ref{prob:Dai-Wang-Yang-Yu-2018}) which are line graph of bipartite planar graphs with sufficiently large maximum degree.
For this purpose, we need to establish the following result which generalizes an exceptional graph in \cite{Zhu-Zhu-2022}.
\begin{thm}\label{thm:gap:largemaximumdegree}
{Let $k$ be a positive integer with $k\ge 4$, and let $a, b, t$ be three positive integers such that $k$ is divisible by $2a$ and $2b$.
Let $G$ be the graph $\overline{K_{t}} \vee (K_{k-a}\cup K_{k-b})$.
If $t\ge \binom{2a}{a}\binom{2b}{b}$, then $G$ is not $k$-choosable.
}\end{thm}
\begin{proof}
{Let $A$, $B$, and $T$ be three disjoint subsets of $V(G)$ with sizes $k-a$, $k-b$, and $t$ corresponding to the vertex sets of three subgraphs $\overline{K_{t}}$, $K_{k-a}$, and $ K_{k-b}$ in $G$.
Let $c_1,\ldots, c_k$ and $c'_1,\ldots, c'_k$ be $2k$ different colors. 
First, assume that $k$ is divisible by $2a$ and $2b$.
We assign the list $\{c_1,\ldots, c_{k}\}$ to all vertices of $A$, and assign the list $\{c'_{1},\ldots, c'_{k}\}$ to all vertices of $B$.
We spilt the first list into $2a$ disjoint subsets $S_1,\ldots, S_{2a}$ with the same size $k/2a$, 
and we spilt the second list into $2b$ disjoint subsets $S'_1,\ldots, S'_{2b}$ with the same size $k/2b$.
Let $\mathcal{I}_a $ be the set of all subsets of $\{1,\ldots, 2a\}$ having size $a$ and 
let $\mathcal{I}_b $ be the set of all subsets of $\{1,\ldots, 2b\}$ having size $b$.
For any pair $I_a,I_b$ with $I_a\in \mathcal{I}_a$ and $I_b\in \mathcal{I}_b$, 
we consider the list $L_{I_a,I_b}$ of size $k$ which is the union of 
$\cup _{i\in I_a}S_i$ and $\cup_{i\in I_b}{S'_i}$. 
We assign this list to one vertex of $T$. 
Sine $|T|\ge \binom{2a}{a}\binom{2b}{b}$, any list of this type can appear on at least one vertex of $T$. 
For the remaining vertices of $T$, we assign arbitrary lists with size $k$ to them.
Suppose, to the contrary, that $G$ admits such a list coloring. 
Since $|A|=k-a$ and $G[A]$ a complete graph, there are exactly $a$ colors $c_{i_1},\ldots, c_{i_a}$ from $\{c_1,\ldots, c_{k}\}$ which did not appear on vertices of $A$. Let $J$ be the set of all $j$ such that $S_j$ does not have any color $c_{i_s}$.
Since $|B|=k-b$ and $G[B]$ is a complete graph, there are exactly $b$ colors $c'_{i'_1},\ldots, c'_{i'_b}$ from $\{c'_1,\ldots, c'_{k}\}$ which did not appear on vertices of $B$. Let $J'$ be the set of all $j$ such that $S'_{j'}$ does not have any color $c_{i'_s}$.
Since $|J|\ge a$ and $J'\ge b$, there exists a pair $I_a, I_b$ satisfying $\mathcal{I}_a \ni I_a\subseteq J$ and $\mathcal{I}_b \ni I_b\subseteq J'$.
Consider the vertex having the list $L$ corresponding to this pair $I_a,I_b$.
This vertex receives a color different from all colors $c_{i_1},\ldots, c_{i_a}$ and $c'_{i'_1},\ldots, c'_{i'_b}$ which is a contradiction. 
Hence the proof is completed.
}\end{proof}
The following corollary shows that the upper bound on $|\Delta(G)|$ in Conjecture~\ref{intro:conj:small-graphs}
 is sharp and cannot be improved by one. (Note that when $k$ is odd and $k\ge 5$, the graph $\overline{K_{5}} \vee (2K_{k-1})$ is not $k$-choosable according to Theorem~\ref{thm:sharpness:conj:t-strongly}). 
\begin{cor}{\rm (\cite{Zhu-Zhu-2022})}\label{cor:2k-2:maximum-degree}
{If $k$ is positive even integer with $k\ge 4$, then the graph $\overline{K_{4}} \vee (2K_{k-1})$ is not $k$-choosable. As a consequence, there exists a $k$-colorable non-$k$-choosable graph $G$ satisfying $\Delta(G)= 2k-2$ and $|V(G)| = 2k+2$.
}\end{cor}
\begin{proof}
{Apply Theorem~\ref{thm:gap:largemaximumdegree} with $a=b=1$ and $t=4=\binom{2a}{a}\binom{2b}{b}$.
}\end{proof}
The following corollary shows that the upper bound on $|V(G)|$ in Conjecture~\ref{intro:conj:small-graphs} 
cannot be larger than $2k+8$ in general. It would be an interesting determine best possible upper bounds.
\begin{cor}\label{cor:2k+9:maximum-degree}
{If $k$ is a positive integer divisible by $4$ and $k\ge 16$, then there exists a $k$-colorable non-$k$-choosable graph $G$ satisfying $\Delta(G)= 2k-3$ and $|V(G)| = 2k+9$.
}\end{cor}
\begin{proof}
{Apply Theorem~\ref{thm:gap:largemaximumdegree} with $a=2$ and $b=1$, and $t=\binom{2a}{a}\binom{2b}{b}=12$.
}\end{proof}
Kim and Park (2015)~\cite{Kim-Park-2015} showed that for square of graphs, the gap between chromatic number and list chromatic number can be arbitrary large. In the following corollary, we introduce a new family of such graphs having bounded maximum degree obtained from line graph of bipartite planar graphs.
\begin{cor}\label{cor:square:graph:maximum-degree-chromatic}
{If $n$ is a positive integer and $k$ is a sufficiently large integer divisible $2n$, there exists a line graph $G$ such that $G^2$ is
 $(k-n+1)$-colorable and non-$k$-choosable graph. In particular, $G$ is the line graph of a bipartite planar graph and $\Delta(G^2) \le 2k-2n$.
}\end{cor}
\begin{proof}
{Let $H_1$ and $H_2$ be two stars with size $k-n$.
 Let $v_1,\ldots, v_t$ be $t$ leaf vertices of $H_1$ and let $w_1,\ldots, w_t$ 
be $t$ leaf vertices of $H_2$. Add new edges $v_iw_i$ to these graphs and call the resulting bipartite planar graph $H$.
Let $G$ be the line graph of $H$. Obviously $G^2$ contains the graph $\overline{K_{t}} \vee (K_{k-n}\cup K_{k-n})$ as a subgraph. Thus by Theorem~\ref{thm:gap:largemaximumdegree}, this graph is not $k$-choosable provided that $t\ge \binom{2n}{n} ^2$.
It is easy to check that this graph is also $(k-n+1)$-colorable, see Figure~\ref{fig:Line_Planar_Bipartite}, 
and also $\Delta(G^2) \le 2k-2n$ provided that $t \le k-n$.
}\end{proof}
\begin{figure}[h]
 \centering
 \includegraphics[scale = 1.6]{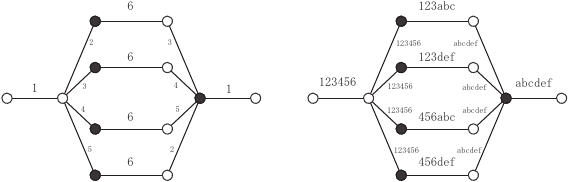}
 \caption{The square of the line graph $L(H)$ is $6$-colorable (left) but not $6$-choosable (right)}
\label{fig:Line_Planar_Bipartite}
\end{figure}
Finally, we propose the following problem to investigate a stronger version for Conjecture~\ref{intro:conj:small-graphs}.
\begin{prob}
{Let $k$ be a positive integer. For every positive integer $i$, what is the maximum number $c(i, k)$ such that if $G$ is a $k$-colorable graph satisfying $\Delta(G)\le 2k-i$ and $|V(G)| \le 2k+c(i,k)$, then $G$ is $k$-choosable?
}\end{prob}
%
%
\section{A revised version to the List Square Coloring Conjecture}

As we already observed, total graphs are square graphs of a class of bipartite graphs.
It is easy to see that the following relation also holds for this family of graphs. 
 $$ \frac{1}{2}\Delta(T(G)) +1\le\Delta(G) +1\le \chi(T(G)).$$
The square of our bipartite graph examples violated this property.
We feel that by replacing this condition, the List Square Coloring Conjecture conjecture can be revised to the following version.
According to the graph construction in Theorem~\ref{thm:bipartite:planar}, 
the lower bound is sharp and cannot be improved by $1/2$. 
\begin{conj}\label{conj:revised}
{If $G$ is a bipartite graph and $\chi(G^2) \ge \frac{1}{2}\Delta(G^2)+1$, then $G^2$ is $1$-strongly chromatic-choosable.
}\end{conj}
One may ask whether this conjecture can be restated for all graphs. If yes, one can easily conclude a weaker version of the List Coloring Conjecture, since $ \frac{1}{2}\Delta(L(G)) +1\le \Delta(G)\le \chi(L(G))$. More generally, for claw-free graphs, one can easily prove that
$ \chi(G) \ge \frac{1}{2}\Delta(G) +1$ (any color around each vertex must appear at most twice). Unfortunately, by the following graph construction, the answer is negative.
\begin{thm}\label{thm:sharpness:conjecture}
{For every integer $k$ with $k\ge 3$, there exists a regular graph $G$ of order $4k-2$ with chromatic number $k$ satisfying $\chi^1_\ell(G) >\chi(G)= \frac{1}{2}\Delta(G)+1$.
}\end{thm}
\begin{proof}
{Let $G$ be a graph with the vertex set $X\cup X' \cup Y'\cup Y$, 
where $X=\{x_1,\ldots, x_{k-1}\}$, $X'=\{x_1,\ldots, x_{k}\}$, $Y'=\{y'_1,\ldots, y'_{k}\}$, $Y=\{y_1,\ldots, y_{k-1}\}$.
We add edges to $G$ such that 
both of $G[X]$ and $G[Y]$ would complete graph, 
both of $G[X,X']$ and $G[Y',Y]$ would be a complete bipartite graph, 
and $G[X',Y']$ would be complete bipartite graph minus edges $x'_iy'_i$, where $1\le i\le k$.
According to the construction, the graph $G$ is a $(2k-2)$-regular graph having some cliques with size $k$.
We claim that $\chi(G)=k$.
To see this, one can color every $x_i$ (resp. $y_i$) by the color $i$ (resp. $i+1$) 
when $1\le i< k$ and color every vertex in $X'$ by the color $k$ and every vertex in $Y'$ by the color $1$. 
Now, we are going to show that $\chi^1_\ell(G)> \chi(G)$. To prove this,
we assign the list $\bar{i}$ to two vertices $x'_{i+k-1}$ and $y'_{i+k-1}$ provided that $1\le i\le k$
and assign the list $\overline{k+1}$ to all vertices of $X\cup Y$, where $\bar{j}=\mathbb{Z}_{k+1}\setminus \{j\}$.
 Suppose, to the contrary, that $G$ admits such a list coloring $c:V(G)\rightarrow \mathbb{Z}_{k+1}$. Therefore,
there exists exactly one color $a\in \mathbb{Z}_{k}$ which is not appeared on all vertices of $X$ and there exists exactly one color $b\in \mathbb{Z}_{k}$ which is not appeared on all vertices of $Y$.
According to the list property, we must have $c(x'_{a+k-1})=k+1$ and $c(y'_{b+k-1})=k+1$ which imply that $a=b$.
Therefore, all other vertices of $X'\cup Y'$ are colored by the color $a$ which is a contradiction, because of $|X'|\ge 3$.
Hence the proof is completed.
 }\end{proof}
\begin{figure}[h]
 \centering
 \includegraphics[scale = 1.3]{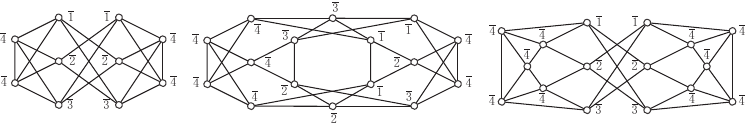}
 \caption{Three $3$-chromatic $4$-regular (non-planar, having only $4$-claws, and planar) graphs $G$ which are not $1$-strongly chromatic-choosable.} 
\label{4-regular}
\end{figure}
However, Conjecture~\ref{conj:revised} is not true for arbitrary graphs, we believe that it can be revised to the following similar version by slightly modifying the lower bound.
\begin{conj}\label{conj:simpler}
{If $G$ is a graph satisfying $\chi(G) \ge \frac{1}{2} (\Delta(G)+3)$, then $G$ is $1$-strongly chromatic-choosable.
}\end{conj}
However, the condition in Conjecture~\ref{conj:revised} is not sufficient for a graph to be $m$-strongly chromatic-choosable (according to Theorem~\ref{thm:gap:largemaximumdegree}), we believe that it can be revised to the following version by increasing the lower bound (depending on $m$). 
\begin{conj}\label{conj:graph:epsilon:c:strongly-choosable}
{Let $m$ be a positive integer. There are two real numbers $\varepsilon_m$ and $c_m$ with $1/2\le \varepsilon_m <1$ such that every graph $G$ satisfying $\chi(G) \ge \varepsilon_m\Delta(G)+c_m$ is $m$-strongly chromatic-choosable.
}\end{conj}
Finally, we propose the following stronger conjecture which says that all of numbers $\varepsilon_m$ and $c_m$ can be bounded with the same number. 
\begin{conj}\label{conj:graph:epsilon:c}
{Let $c_0$ be a nonnegative integer. There are two real numbers $\varepsilon$ and $c$ with $1/2\le \varepsilon <1$ such that if a graph $G$ satisfies $\chi(G) \ge \varepsilon \Delta(G)+c$, then $\chi_{\ell}(G) \le \chi(G)+c_0$.
}\end{conj}
\subsection{Comparing list-chromatic measures}
As we already mentioned that the inequalities $\chi(G) =\chi^0_\ell(G) \le \chi^1_\ell(G) \le \cdots \le \chi^\infty_\ell(G) =\chi_\ell(G)$ hold for these list chromatic measures. We shall below introduce a stronger relation between these parameters.
\begin{thm}\label{thm:i-strongly}
{Let $m$ be a positive integer. If $G$ is a graph, then
$$ \chi^{m-1}_\ell(G)\le \chi^{m}_\ell(G)\le \chi^{m-1}_\ell(G)+\chi(G)-1\le (m+1)\chi(G)-m.$$
In addition, $\min \{\chi(G)+(1-\frac{1}{m+1})p,(m+1)\chi(G)\}\le \chi^{m}_\ell(G)+m$ for complete multipartite graphs with part size at least $\binom{p}{m}$.
}\end{thm}
\begin{proof}
{Let $k=\chi^{m-1}_\ell(G)$, and let $B$ and $C$ be two disjoint sets of colors with size $k-1+m$ and $\chi(G)$.
Let $L:V(G)\rightarrow 2^{B\cup C}$ be a mapping such that for each vertex $v$, $|L(v)|= k-1+\chi(G)$ and $\cup_{v\in V(G)}L(v) \subseteq B\cup C$.
By the definition of $\chi(G)$, there is a coloring $c_0:V(G)\rightarrow C$ of $G$. 
For each vertex $v$, define $L_0(v)=B$, if $c_0(v) \in L(v)$, and define $L_0(v)=L(v)\setminus C$ if $c_0(v) \not \in L(v)$. 
Note that $|L_0(v)| \ge k$ and $|\cup_{v\in V(G)}L_0(v)|\le k+m-1$.
Therefore, by the definition of $\chi^{m-1}_\ell(G)$, there is a coloring $c'_0:V(G)\rightarrow B$ 
of $G$ such that for each vertex $v$, $c'_0(v) \in L_0(v)$ and so $c'_0(v)\neq c_0(v)$.
Now, for each vertex $v$, 
we define $c(v)=c_0(v)$, if $c_0(v) \in L(v)$, and define $c(v)=c'_0(v)$ if $c_0(v) \not \in L(v)$.
Thus $c(v) \in L(v)$ regardless of $c(v)=c_0(v)$ or $c(v)=c'_0(v)$. Note that since $B$ and $C$ are disjoint, adjacent vertices receive different colors from $c$.
These imply that $ \chi^{m}_\ell(G)\le k-1+\chi(G)$, which completes the first part of the proof.

We shall below write $k$ for $\chi^{m}_\ell(G)$. Let $G$ be the complete multipartite graph with part size $\binom{p}{m}$.
Let $S$ be the set of all lists with size $m$ of elements of $\mathbb{Z}_{p}$ and 
let $L:V(G)\rightarrow S$ be a mapping such that every part of $G$ receives all lists $\{1,\ldots, k+m\}\setminus L$, where $L\in S$. 
Assume that $G$ has such a list coloring $c:V(G)\rightarrow \{1,\ldots, k+m\}$ so that 
 $t$ parts receive at least one color from $\{p+1,\ldots, k+m\}$ and $t_p$ parts receive colors only from $\{1,\ldots, p\}$.
According the list property, every part of the second type must receive at least $m+1$ colors from $\{1,\ldots, p\}$.
This implies that $t_p \le p/(m+1)$.
Thus the number of used colors must be at least $t+(m+1)t_p$.
If $t_p=\chi(G)$, then this lower bound is exactly $(m+1)\chi(G)$.
If $t_p<\chi(G)$, then at least $ \chi(G)-t_p$ colors used from $\{p+1,\ldots, k+m\}$ and so $k+m-p\ge \chi(G)-t_p$.
These can imply that $ \chi^{m}_\ell(G)+m\ge \chi(G)+(1-\frac{1}{m+1})p$.
Hence the proof is completed.
}\end{proof}
For square of graphs $G$, the gap between chromatic number $\chi(G^2)$ and the list chromatic number $\chi^1_\ell(G^2)$ can be arbitrary large by the following corollary.
\begin{cor}{\rm (\cite{Kim-Park-2015})}
{ There exists a graph $G$ satisfying $\chi^1_\ell(G^2)-\chi(G^2)\ge n$, where $n$ is an arbitrary given positive integer.
}\end{cor}
\begin{proof}
{Let $p$ be an arbitrary prime number. It is known that there exists a graph whose square is the complete multipartite graph $G$ with part size $p$ and with chromatic number $2p-1$~\cite{Kim-Park-2015}.
By Theorem~\ref{thm:i-strongly}, we must have $\chi(G^2)+p/2 \le \chi^1_\ell (G^2)+1$.
This can complete the proof.
}\end{proof}
%
%
%
%
%
%
%
%
%
\section{Partial solutions to Conjectures~\ref{intro:conj:small-graphs} and~\ref{conj:simpler}: graphs of small order}

\subsection{Non-chromatic-choosable graphs of small order: computer reports}
Erd\H os, Rubin, and Taylor (1980)~\cite{Erdos-Rubin-Taylor-1980} showed that there are only two edge-minimal non-$2$-choosable graphs of order $6$. More generally, they characterized all non-$2$-choosable graphs (Figure~\ref{smallest-non-chromatic-choosable-graph} shows edge-minimal non-$2$-choosable graphs having order at most $8$). In 2019 Nelsen \cite[Page 40]{Nelsen-2019} showed that all graphs of order at most $8$ are $3$-choosable using a computer search. 
He also posed a question about the smallest graph with chromatic number $3$ which is not $3$-choosable.
Recently, Zhu and Zhu (2022) discovered all edge-minimal $k$-chromatic non-$k$-choosable graphs of order at most $2k+2$.

\begin{figure}[h]
 \centering
 \includegraphics[scale = 1]{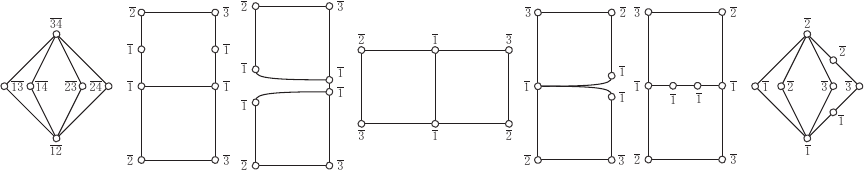}
 \caption{All edge-minimal non-chromatic-choosable graphs $G$ of order at most $8$: 
$\Delta(G)\in \{3,4\}$ and $|E(G)|\in \{7,8,9,10\}$.} 
\label{smallest-non-chromatic-choosable-graph}
\end{figure}
\begin{thm}{\rm (\cite{Zhu-Zhu-2022})}
{For all even integers $k$ with $k\ge 4$, the graphs $\overline{K_{4}} \vee (2K_{k-1})$ or $K_{1*\frac{k-2}{2},3*\frac{k+2}{2}}$ 
are all edge-minimal $k$-chromatic non-$k$-choosable graphs of order at most $2k+2$.
}\end{thm}
For odd positive integers $k$, they had already proved in \cite{Zhu-Zhu-arXiv} that there is no $k$-chromatic non-$k$-choosable graph of order $2k+2$ which was originally conjectured by Noel (2013) \cite{Noel-2013}. On the other hand, the complete multipartite graph $K_{5,2*(k-1)}$ is a $k$-chromatic non-$k$-choosable graph of order $2k+3$~\cite[Proposition 5]{Enomoto-Ohba-Ota-Sakamoto-2002}.
 By an innovative computer search, we succeeded to completely characterize edge-minimal $3$-colorable non-$3$-choosable graphs of order at most $9$ and observed that the smallest graph contains $19$ edges (the top-right graph in Figure~\ref{9all}). It was already proved that $K_{3,3,3}$ is a $3$-chromatic non-$3$-choosable graph of order $9$ \cite[Theorem 3]{Kierstead-2000}. 
\begin{thm}\label{thm:9}
{The complete multipartite graphs $K_{5,2,2}$, $K_{3,3,3}$, and $K_{4,4,1}$ are all edge-maximal $3$-colorable non-$3$-choosable graphs of order $9$ and the graphs in Figure~\ref{9all} are all edge-minimal $3$-colorable non-$3$-choosable graphs of order at most $9$.
}\end{thm}
\begin{proof}
{We used a computer search to find bad list assignments $L$ of size $3$ on a given complete $3$-partite graph $G$ having no $L$-coloring 
(based on Corollary 5 in \cite{Kierstead-2000} to reduce the number of required colors). Next, we started to delete edges one by one to discover edge-minimal $3$-chromatic non-$3$-choosable graphs. 
}\end{proof}
\begin{figure}[h]
 \centering
 \includegraphics[scale =1.35]{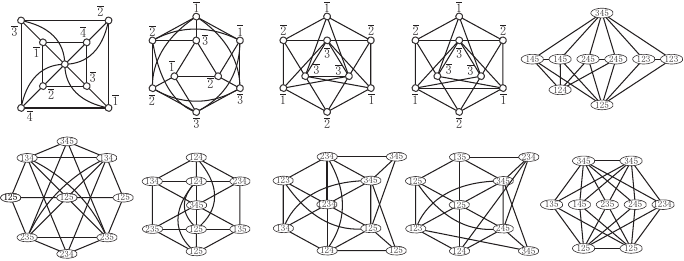}
 \caption{All edge-minimal $3$-chromatic non-$3$-choosable graphs $G$ of order $9$: 
$\Delta(G)\in \{5,6,8\}$ and $|E(G)| \in \{19,20,21, 22\}$.} 
\label{9all}
\end{figure}
 For odd integers $k$ with $k\ge 5$, we conjecture that the graph $\overline{K_{5}} \vee (2K_{k-1})$ is the smallest $k$-chromatic non-$k$-choosable graph (this graph is not $k$-choosable according to Theorem~\ref{thm:sharpness:conj:t-strongly}). 
Moreover, we put forward the following stronger assertion.
\begin{conj}\label{conj:smallest}
{Let $k$ be an odd integers $k$ with $k\ge 5$. 
The two graphs $\overline{K_{5}} \vee (2K_{k-1})$ or $K_{1*\frac{k-3}{2},3*\frac{k+3}{2}}$ are all 
edge-minimal $k$-chromatic non-$k$-choosable graphs of order at most $2k+3$.
}\end{conj}

For even $k$, it would be an interesting but difficult problem to characterize all edge-minimal $k$-chromatic non-$k$-choosable graphs of order $2k+3$ to settle both Conjecture~\ref{intro:conj:small-graphs} and~\ref{conj:smallest}. So, we write the following question for this purpose.
\begin{prob}
{Characterize all edge-minimal (or edge-maximal) $k$-chromatic non-$k$-choosable graphs of order at most $2k+3$ (and also $2k+4)$.
}\end{prob}

It seems that the problem of characterizing all edge-minimal $k$-chromatic non-$k$-choosable graphs of order $2k+4$ is hard. 
In fact, for the small case $k=3$, there are (precisely) $827$ edge-minimal $3$-chromatic non-$3$-choosable graphs of order at most $10$. But for order $2k+3$, the situation seems easier whether $k$ is odd or even. By a computer search, we observed that there are only $11$ edge-minimal $4$-chromatic non-$4$-choosable graphs of order at most $11$, see Figures~\ref{fig:at-most-11-all}. For $k=5$, we also observed that there are only two edge-minimal $5$-chromatic non-$5$-choosable graphs $\overline{K_{5}} \vee (2K_{4})$ and $K_{1,3*4}$ of order at most $13$ using at most $9$ colors in lists, see Figure~\ref{fig:all:X=5-n=13}.
\begin{figure}[h]
 \centering
 \includegraphics[scale =.92]{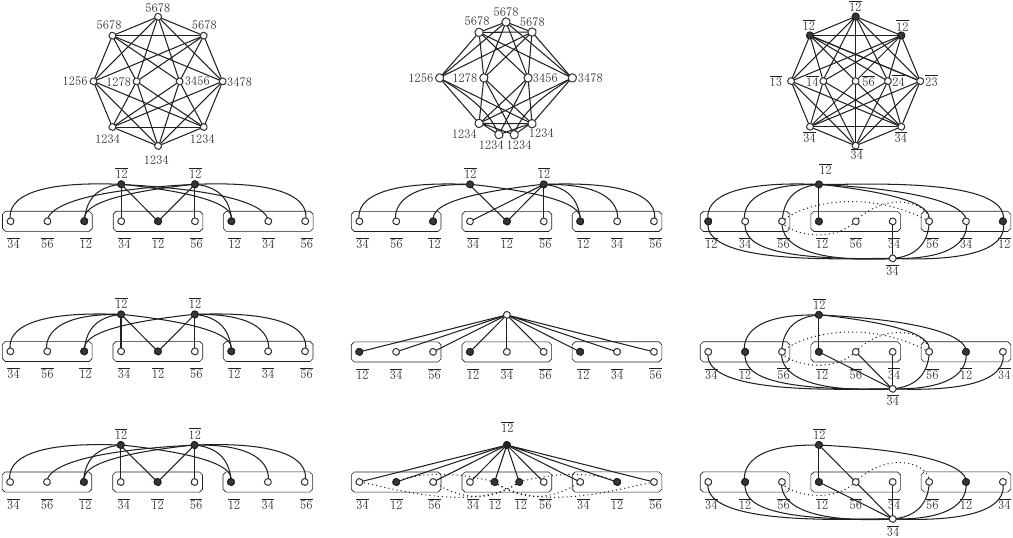}%
 \caption{These graphs are all edge-minimal $4$-chromatic graphs of order at most $11$ which are not (reps. $2$-strongly) $4$-choosable
except from the top-right graph (resp. the top-left and top-middle graphs). 
Three rectangles are corresponded to partite sets of an induced complete multipartite subgraph $K_{3,3,3}$ or $K_{3,3,4}$, and dotted lines mean deleting those edges from it. 
} 
\label{fig:at-most-11-all}
\end{figure}

\begin{figure}[h]
 \centering
 \includegraphics[scale =1.3]{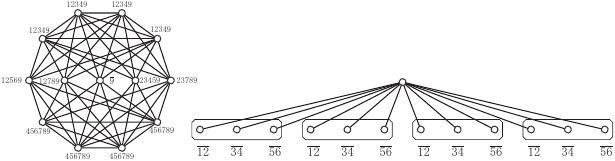}%
 \caption{Two edge-minimal $5$-chromatic non-$5$-choosable graphs of order at most $13$; four rectangles are corresponded to partite sets of the induced complete multipartite subgraph $K_{3,3,3,3}$.} 
\label{fig:all:X=5-n=13}
\end{figure}
%
%
%
%
%
%
%
%
\subsection{$k$-chromatic graphs of order less than $3k$ are $1$-strongly $k$-choosable}
In this section, we are going to confirm Conjecture~\ref{conj:simpler} for $k$-chromatic graphs of small order by proving that all $k$-chromatic graphs of order less than $3k$ are $1$-strongly $k$-choosable (the square graph in Theorem~\ref{thm:claw-free:cubic} shows that the condition on the order is sharp even for square graphs). For this purpose, we need to recall the following well-known lemma.
\begin{lem}{\rm (Hall \cite{Hall-1935})}\label{thm:Hall}
{Let $G$ be a bipartite multigraph with bipartition $(A,B)$. Then $G$ admits a matching $M$ covering all vertices of $A$ if and only if for every 
$S\subseteq A$, $|S| \le |N_G(S)| $. In addition, $M$ covers at least $|A|-t$ vertices of $A$ if and only if for every 
$S\subseteq A$, $|S|\le |N_G(S)|+t$, where $t$ is a nonnegative integer.
}\end{lem}
\begin{proof}
{By applying the first assertion to the graph obtained from $G$ by joining $t$ new vertices to all vertices of $A$, we can derive the second assertion immediately.
}\end{proof}
By a computer search, we observed that for 
for $k=3$ (resp. $k=2$), there are only four (resp. one) edge-minimal graphs of order $3k$ that are not $1$-strongly $3$-choosable. These graphs are depicted in Figure~\ref{9all} by the four top-left graphs (resp. the middle graph in Figure~\ref{smallest-non-chromatic-choosable-graph}). In addition, for $k=4$ (resp. $k=5$), there are only $14$ (resp. $86$) edge-minimal graphs of order $3k$ that are not $1$-strongly $k$-choosable. 

\begin{thm}\label{thm:1-strongly-choosable}
{Every $k$-chromatic graph $G$ of order at most $3k-1$ is $1$-strongly $k$-choosable.
In addition, if $|V(G)|=3k$ and $G$ is not $1$-strongly $k$-choosable, then one of the following conditions holds:
\begin{enumerate}{

\item
The graph $G$ is obtained from the complete $k$-partite graph $K_{(k+1)*2, 1*(k-2)}$ by removing some of the edges $v_{1,j}v_{2,j}$ between any two parts $V_i=\{v_{i, 1}, \ldots, v_{i, k+1}\}$ of size $k+1$ for which $i=1,2$ provided that $k\ge 4$. 

\item 
The graph $G$ is obtained from the complete $k$-partite graph $K_{3*k}$ by removing at most two edges $v_{i, t}v_{j, t}$ between any two parts $V_i$ and $V_j$. More precisely, for every nonempty set of indices $I\subseteq \{1,\ldots, k\}$, $\sum_{1\le j \le 3}\chi(G[U_j])\ge |I|+2$, 
where 
$U_j=\{v_{i, j}: v_i \in I\}$ and
 $V_i=\{v_{i, 1}, v_{i, 2}, v_{i, 3}\}$ is a part of $K_{3*k}$.
}\end{enumerate}
As a consequence, we must have $\Delta(G)\ge \frac{7}{2}(k-1)$.
}\end{thm}
\begin{proof}
{Let $G$ be a complete $k$-partite graph 
with $k$-partition $V_1,\ldots, V_k$ and for each vertex $v$, let $L(v)$ be a subset of $\{c_1,\ldots, c_{k+1}\}$ of size $k$. Let $\mathcal{B}$ be a bipartite graph with bipartition $(P,C)$, where $P=\{p_1,\ldots, p_{k}\}$ and $C=\{c_1,\ldots, c_{k+1}\}$. We join $p_i$ to $c_j$ in $\mathcal{B}$ if $L(v)\neq \{c_1,\ldots, c_{k+1}\} \setminus \{c_j\}$ for all $v\in V_i$.
First, we claim that $G$ admits an $L$-coloring if and only if 
the graph $\mathcal{B}$ contains a matching $M$ covering at least $k-1$ vertices of $P$.
Assume $\mathcal{B}$ contains a matching $M$ covering at least $k-1$ vertices of $P$. 
We may assume that $|E(M)|=k-1$. 
Thus $d_M(p_{t})=d_M(c_{j_1})= d_M(c_{j_2})=0$ for an index $t$ and two indices $j_1$ and $j_2$.
Let $v\in V(G)$. 
If $v\in V_{t}$, then we color this vertex by $c_{j_1}$ or $c_{j_2}$ which is in $L(v)$.
If $v\in V_{i}$ and $i\neq t$, 
then we color this vertex by the color $c_j$ where $p_ic_j\in E(M)$.
This yields that $G$ admits an $L$-coloring.
Conversely, assume that $G$ contains an $L$-coloring. 
Since $G$ is a complete $k$-partite graph and the number colors is $k+1$, there must be at least $k-1$ parts having exactly one color. This means that $\mathcal{B}$ contains a matching $M$ covering at least $k-1$ vertices of $P$. Hence the claim is proved.

Now, assume that $H$ does not contain a matching $M$ covering at least $k-1$ vertices of $P$.
Thus by Lemma~\ref{thm:Hall}, there exists a subset $S\subseteq P$ satisfying $|N_\mathcal{B}(S)| \le |S|-2$.
Let $p$ be a vertex in $S$ with maximum degree.
It is not hard to check that 
$d_\mathcal{B}(p).|S|\ge \sum_{p_i\in S}d_\mathcal{B}(p_i)\ge
 \sum_{p_i\in S}(k+1-|V_i|) = (k+1)|S|-\sum_{p_i\in S}|V_i|\ge
k(|S|+1)-|V(G)|$. 
Therefore, $|S|\ge |N_\mathcal{B}(S)|+2\ge d_\mathcal{B}(p)+2\ge k+2+\frac{1}{|S|}(k-|V(G)|)$ 
and so $|V(G)| \ge k+(k+2-|S|)|S|$. Assume that $V(G)\le 3k$.
Since $2\le |S|\le k$, these inequalities imply that $|S|\in \{2,k\}$, $|V(G)|= 3k$, $|V_i|=1$ for all $p_i\in P\setminus S$, 
and $N_\mathcal{B}(S) = N_\mathcal{B}(p_i)$ and $|N_\mathcal{B}(p_i)|=k+1-|V_i|$ for all $p_i\in S$.
Hence we have the following two cases:

{\bf Case 1}: $|S|=k$.

In this case, according to the above-mentioned arguments, we must have $|V_i|=3$ for all $i\in \{1, \ldots, k\}$. 
In addition, there must be three fixed types of lists distributed on the triple vertices of any $V_i=\{v_{i, 1}, v_{i, 2}, v_{i, 3}\}$.
We may assume that the given list of $v_{i, j}$ is $ \bar{j}=\mathbb{Z}_{k+1 }\setminus \{j\}$, where $j\in \{1,2,3\}$.
Suppose, to the contrary, $G$ does not contain an edge in which 
$e=v_{a,a'}v_{b,b'}$ for which $a\neq a'$ and $b'\neq b'$.
Without loss of generality, we may assume that $a=k-1, a'=2$ and $b=k, b'=3$.
In this case, we color every vertex $V_i$ by the color $i$ provided that $1\le i\le k-2$ .
Next, we color both vertices $v_{k-1, 2}$ and $v_{k, 3}$ by $1$, 
both vertices $v_{k-1, 1}$ and $v_{k-1, 3}$ by $2$, 
and both vertices $v_{k, 1}$ and $v_{k, 2}$ by $3$.
Thus $G$ admits an $L$-coloring, which is a contradiction.
To complete the proof of Case 1, it is enough to prove the following claim.
Note that since $G$ does not admit an $L$-coloring, by setting $I=\{i, j\}$ and $G'=G$, one can conclude that there must be at leas one edge $v_{i, t}v_{j, t}$ between $V_i$ and $V_j$ (because the chromatic number of at least one set $W_j$ is at least two).

{\bf Claim}: 
Every spanning subgraph $G'$ of $G$ admits an $L$-coloring if and only if $\sum_{1\le j\le 3}\chi(G'[U_j])\le |I|+1$ for every index-set $I\subseteq \{1,\ldots, k\}$, where $U_j=\{v_{i,j}: i\in I\}$.

{\bf Proof of claim}: 
Let $I$ be an arbitrary subset of $\{1,\ldots, k+1\}$.
Let $A_j$ be a subset of $\mathbb{Z}_{k+1 }$ of size $\chi(G[U_j])$ excluding the color $j$, 
where $j\in \{1,2, 3\}$. 
We may assume that $A_1$, $A_2$, and $A_3$ are disjoint.
If $|A_1\cup A_2\cup A_3|$ is at most $k+1-|I|$, then we color any triple of vertices of every part $V_i$ with $i \in \{1,\ldots, k+1\}\setminus I$ by a unique color in $\mathbb{Z}_{k+1 } \setminus (A_1\cup A_2\cup A_3)$.
Next, we provide a proper coloring for $G'[U_i]$ using all colors of $A_i$.
Thus $G'$ admits an $L$-coloring, as desired.

Conversely, assume that $G'$ admits an $L$-coloring. 
Let $I$ be the set of all indices $i$ such that all vertices of $V_i$ have the same color. Obviously, $|I|\ge 2$.
 We consider such a coloring with minimum number of colors appearing on exactly two vertices of a part $V_i$.
To complete the proof, we are going to show that this minimum number is zero.
Otherwise, without loss of generality, we may assume that a color 
$c'$ appeared on only two vertices $v_{1, 2}$ and $v_{1, 3}$ of a part $V_1$.
If $c'\not \in \{1,2, 3\}$, then we replace the color of $v_{1,1}$ by $c'$ to obtain a new $L$-coloring with smaller desired number which is a contradiction.
Assume that $c' \in \{1,2, 3\}$ which implies that $c'=1$.
 Let $i\in I\setminus \{1\}$. 
If the color of $v_{i, 2}$ is not in $\{1,2,3\}$, then we can exchange it with the color $1$ on all vertices and replace the color of $v_{1,1}$ by this color. This gives a new $L$-coloring with smaller desired number which is a contradiction.
Thus the color of $v_{i, 2}$ must be $3$. Similarly, the color of $v_{i, 3}$ must be $2$. These imply that the color of $v_{1,1}$ is not in $\{1,2,3\}$.
Finally, we exchange the color $3$ by the color of $v_{1,1}$ on all vertices and alternatively, 
we exchange the color $3$ by the color $1$ on all vertices. 
Again, this gives a new $L$-coloring with smaller desired number which is a contradiction.
Hence the proof of the claim and Case 1 are completed.

{\bf Case 2}: $|S|=2$.

In this case, according to the above-mentioned arguments, we must have $|V_1|=|V_2|=k+1$ and $|V_3|=\cdots =|V_k|=1$ (by permuting parts).
In this case, there must be $k+1$ different types of lists distributed on $k+1$ vertices of $V_1$ (resp. $V_2$). 
Let $V_i=\{v_{i, 1}, \ldots, v_{i, k+1}\}$ for which $i=1,2$.
We may assume that the given lists of both vertices $v_{1, j}$ and $v_{2, j}$ are the same list $ \bar{j}=\mathbb{Z}_{k+1 }\setminus \{j\}$, where $j\in \{1,\ldots, k+1\}$.
If $k\le 2$, then the assertion follows from Case 1. So, suppose $k\ge 3$.
Suppose, to the contrary, $G$ does not contain an edge in which 
(a1) $e=v_{1,j}v_{2,j'}$ and $j\neq j'$,
(a2) $e=uv$ for which $V_j=\{v\}$, and $V_j'=\{u\}$, or
(a3) $e=v_{i,j}v$ for which $V_{j'}=\{v\}$.

In first subcase, we first color all $k-2$ vertices in $V_3\cup \cdots\cup V_k$ using $k-2$ colors in $\{1,\ldots, k+1\}\setminus \{j, j'\}$
by avoiding the forbidden color for every vertex. Let $c'$ be the unique remaining color in this set. 
Next, we color both ends of $e$ with the same color $c'$. 
Finally, we color every remaining vertex of $V_1$ by $j$ and 
every remaining vertex of $V_2$ by $j'$. Thus $G$ admits an $L$-coloring, which is a contradiction.

In second subcase, we first color all $k-4$ vertices in $\cup_{i\in \{3,\ldots, k\}\setminus \{j, j'\}} V_i$ using $k-4$ colors in $\{1,\ldots, k+1\}$ by avoiding the forbidden color for every vertex. 
Next, we color both ends of $e$ by the same remaining color different from forbidden colors of them.
 Let $c_1,\ldots, c_4$ the remaining colors.
Finally, we color every vertex of $V_1$ by $c_1$ or $c_2$ and 
every vertex of $V_2$ by $c_3$ or $c_4$ by avoiding the forbidden color for every vertex. Thus $G$ admits an $L$-coloring, which is a contradiction.

In third subcase, we first color all $k-3$ vertices in $\cup_{i\in \{3,\ldots, k\}\setminus \{j'\}} V_i$ using $k-3$ colors in 
$\{1,\ldots, k+1\}\setminus \{j\}$ by avoiding the forbidden color for every vertex. 
Next, we color both ends of $e$ by the same remaining color different from forbidden colors of them.
Next, we color every remaining vertex of $V_i$ by $j$.
 Let $c_1,c_2$ be the remaining colors.
Finally, we color every vertex of $V_{3-i}$ by $c_3$ or $c_4$ by avoiding the forbidden color for every vertex. Thus $G$ admits an $L$-coloring, which is a contradiction.

Now, we claim that if $G_0$ is the spanning subgraph of $G$ excluding all edges $v_{1, j}v_{2, j}$, then 
$G_0$ does not admit an $L$-coloring. 
Otherwise, at most three colors $c_1,\ldots, c_3$ appeared on all vertices of $V_1\cup V_2$.
According to the list assignments, for every $V_i$, there are at least two colors appeared on it.
If these three colors appeared on vertices of a part $V_i$, then each of those colors appears on at most on vertex of $V_{3-i}$ and so
 $k+1=|V_{3-i}|\le 3$, which is impossible.
 Thus for every $V_i$, there are exactly two colors appeared on it.
On the other hand, there is a color $c_t$ appeared on both $V_1$ and $V_2$, and so it must appear on only two vertices $v_{1,j}$ and $v_{2,j}$. 
Without loss of generality, we may assume that $t=3$ and $c_i$ appeared on $V_i$.
If $c_1\neq j$ (resp. $c_2\neq j$), then the vertex $v_{1, c_1}$ (resp. $v_{2, c_2}$) receives a color different from $c_1, c_2, c_3$ which is a contradiction.
}\end{proof}
Motivated by Theorem~\ref{thm:1-strongly-choosable}, we put forward the following conjecture for further investigation. We succeeded to prove this conjecture for $t=2$ with a stronger version by increasing the upper bound by one and charactering all exceptional graphs.
More precisely, there is a unique exceptional graph when $k$ is even and $k\ge 4$, and 
there are $21$ edge-minimal exceptional graphs when $k=5$.
We will present its details in a forthcoming paper.
\begin{conj}\label{conj:t-strongly-V(G):bound}
{Let $k$ and $t$ be two positive integers. If $G$ is a $k$-chromatic graph and $|V(G)|< 2k+\lceil \frac{k}{t}\rceil$, then $G$ is $t$-strongly $k$-choosable. 
}\end{conj}
 It would be an interesting problem to determine the exceptional graphs with higher order. 
In the following theorem, we introduce several families of $k$-colorable graphs which are not $t$-strongly $k$-choosable. 
This result improves and generalizes graph construction of Proposition 5 and Example 1 in~\cite{Enomoto-hba-Ota-Sakamoto-2002}.
\begin{thm}\label{thm:sharpness:conj:t-strongly}
{If $k$ and $t$ are two integers and $2\le t\le k$, then the following graphs are not $t$-strongly $k$-choosable:
\begin{enumerate}{
\item [$\bullet$]
 $K_{ \lceil\frac{k+t}{t} \rceil *n,1*(k-n)}$, where $n=t+1\le k$.

\item [$\bullet$]
$K_{ \lfloor \frac{k+t}{t} \rfloor *n,1*(k-n)}$, where $n =t+1+r\le k$ and $k\stackrel{t}{\equiv}r\in \{0,\ldots, t-1\}$. 

\item [$\bullet$]
$\overline{K_{ \lceil\frac{k+2}{t}\rceil+3}} \vee (2K_{k-1})$ when $t$ is odd, and
$\overline{K_{ \lceil\frac{k}{t}\rceil+3}} \vee (2K_{k-1})$ when $t$ is even. 

}\end{enumerate}
}\end{thm}
\begin{proof}
{Let $V_1, \ldots, V_k$ be a partition of $V(G)$ corresponding to a $k$-coloring of $G$. 
First assume $G=K_{ \lceil\frac{k+t}{t} \rceil *n,1*(k-n)}$ and $V_i=\{v_{i, 1}, \ldots,v_{i, \lceil\frac{k+t}{t} \rceil} \}$ when
 $1\le i \le n=t+1$. 
Let $C_1,\ldots, C_p$ be a partition of $\{1,\ldots, k+t\}$ such that $|C_p|\le t$ and $|C_i|=t$ for all $i$ with $1\le i< p$.
For each $v_{i, j}\in V_i$, we assign the list $L(v_{i, j})=\{1,\ldots, k+t\}\setminus C_j$.
For any $v\in V_i$ with $i> n$, we assign the list $\{1,\ldots, k+t\}$.
If $G$ admits an $L$-coloring, then at least two colors must appear on vertices of any part $V_i$ with $1\le i\le n$,
 because every color of $\{1,\ldots, k+t\}$ is forbidden from the list of a vertex in $V_i$.
 Thus we need at least $2n+(k-n)$ colors which is strictly greater than $k+t$. This is a contradiction.

Next, assume $G=K_{ \lfloor \frac{k+t}{t} \rfloor *n,1*(k-n)}$ and $V_i=\{v_{i, 1}, \ldots,v_{i, \lfloor \frac{k+t}{t} \rfloor} \}$ 
when $1\le i \le n = t+1+r$. 
Let $C_1,\ldots, C_p$ be a partition of $\{1,\ldots, k+t-r\}$ such that $|C_p|\le t$ and $|C_i|=t$ for all $i$ with $1\le i< p$.
For each $v_{i, j}\in V_i$, we assign the list $L(v_{i, j})=\{1,\ldots, k+t\}\setminus C_j$.
For any $v\in V_i$ with $i> n$, we assign the list $\{1,\ldots, k+t\}$. The right graph in Figure~\ref{fig:all:X=5-n=13} shows the special case $k=5$ and $t=2$.
If $G$ admits an $L$-coloring, then at least two colors must appear on vertices of any part $V_i$ with $1\le i\le n$ 
except from at most $r$ parts used only one color from $\{i: k+t-r < i\le k+t\}$, 
because every color of $\{1,\ldots, k+t-r\}$ is forbidden from the list of a vertex in $V_i$.
 Thus we need at least $k+t+1$ colors which is again a contradiction. 

Finally, assume that $G=\overline{K_{n}} \vee (2K_{k-1})$, where 
$n-4 = \lceil\frac{k-t+2}{t}\rceil$ when $t$ is odd and $n-4 = \lceil\frac{k-t}{t}\rceil$ when $t$ is even.
Let $V_1$ and $V_2$ be the vertex sets of the subgraphs isomorphic to $K_{k-1}$, and
let $U$ be the vertex set of the subgraph isomorphic to $\overline{K_n} $.
Assume that $V_i=\{v_{i, 1},\ldots, v_{i, k-1}\}$ and $U=\{u_{i,1},\ldots, u_{i,n}\}$.
Let $A_i=\{a_{i, 1},\ldots, a_{i, t}\}$ and $B=\{b_1, \ldots, b_{k-t}\}$ be the three disjoint sets of colors, where $i=1,2$.
We define $A_{i, 1}=\{a_{i, j}: 1\le j\le \lfloor t/2\rfloor \}$ and $A_{i, 2}=\{a_{i, j}: \lfloor t/2\rfloor < j\le t-t_0 \}$, where $t\stackrel{2}{\equiv}t_0\in \{0,1\}$. Let $B'=B$ when $t$ is even and let $B'=B\cup \{a_{1,t}, a_{2,t}\}$ when $t$ is odd.
Let $C_1,\ldots, C_p$ be a partition of $B'$ such that $|C_p|\le t$ and $|C_i|=t$ for all $i$ with $1\le i< p$. 
In addition, one set $C_i$ includes both of vertices $a_{1,t}, a_{2,t}$ when $t$ is odd; notice that this is possible because $t\neq 1$.
First, we assign the list $A_i\cup B$ to every vertex of $V_i$.
Next, for each $u_{j}\in U$ with $j\le n-4$, we assign the list $L(u_{ j})=(A_1\cup A_2\cup B)\setminus C_j$.
Note that there are four remaining vertices of $U$.
If $t$ is odd, we assign four lists
 $A_{1, i}\cup A_{2, i} \cup \{a_{1, t}\}$ and $A_{1, i}\cup A_{2, j} \cup \{a_{2, t}\}$ to these vertices, where $i,j\in \{1,2\}$ and $j\neq i$.
If $t$ is even, we assign four lists 
$A_{1, i}\cup A_{2, j} $ to these vertices, where $i, j\in \{1,2\}$. The left graph in Figure~\ref{fig:all:X=5-n=13} shows the special case $k=5$ and $t=4$.

Suppose, to the contrary, that $G$ admits an $L$-coloring. 
Then all vertices of $V_i$ receive all colors from $A_i\cup B$ except a color $x_i$.
If $x_i\in B$ for an index $i\in \{1,2\}$, then we must have $x_1=x_2$ (otherwise, all colors are forbidden to use for vertices in $U$, because vertices of $V_1\cup V_2$ are connected to vertices of $U$). This derives contradiction, because there is a vertex in $U$ such that the color $x_1$ is forbidden in the list of $u$ and so this vertex receives a color of $B$ different from $x_1$, which is impossible.
Thus we may assume that $x_1\in A_1$ and $x_2\in A_2$. 
In this case, there is a vertex in $u\in U$ such that both colors $x_1, x_2$ are forbidden in the list of $u$ and so this vertex receives a color of $A_1\cup A_2\cup B$ different from $x_1$ and $x_2$. 
This derives a contradiction because $u$ is adjacent to all vertices of $V_1\cup V_2$ and its color cannot appear on all vertices of them.
Hence the proof is completed. 
}\end{proof}
%
%
%
%
%
%
%
%
\section{The existence of non-chromatic-choosable graphs with bounded maximum degree}
\label{sec:last}
\subsection{Bipartite graphs}
Motivated by Conjecture~\ref{conj:graph:epsilon:c}, we would like to know upper bounds on the maximum degrees to guarantee $k$-choosability property of $b$-coloring graphs. 
According to Theorems~\ref{thm:gap:largemaximumdegree} and~\ref{thm:Kn-n:smaller-graph} (with setting $a=\log(k)$), we have the inequalities $ f(k,k)\le 2k-O(\log(k))$ and $f(k,k)\le f(k,k-1) \le \cdots \le f(k,2)< 2^{k}$.
\begin{prob}
{Let $k$ and $b$ be two positive integers with $k\ge b$. What is the maximum number $f(k,b)$ such that every $b$-colorable graph $G$ with maximum degree at most $ f(k,b)$ is $k$-choosable?
}\end{prob}
In 1992 Alon and Tarsi showed that there is an interesting relation between choosability and orientation of bipartite graphs as the following theorem.
They also remarked that the following result is sharp for complete bipartite graphs $K_{n,n^n}$. 
 This result also implies that every bipartite graph $G$ of maximum degree at most $\Delta $ is $\lceil \Delta/2 \rceil$-choosable~\cite[Theorem 3.2]{Alon-Tarsi-1992}. 
In addition, there are non-$3$-choosable bipartite graphs with maximum degree $5$ \cite[Proposition 6]{Bessy-Havet-Palaysi-2002} while every bipartite graph with maximum degree $4$ must be $3$-choosable (which implies that $f(3, 2)= 5$). Note that Alon and Krivelevich (1998) \cite{Alon-Krivelevich-1998} conjectured that every bipartite graph $G$ of maximum degree at most $\Delta $ is $O(\log \Delta)$-choosable (which means that $f(k, 2)\ge 2^{O(k)}$).
\begin{thm}{\rm (\cite{Alon-Tarsi-1992})}\label{thm:Alon-Tarsi-1992}
{Let $G$ be a bipartite graph and let $D$ be an orientation of $G$, and $L:V(G)\rightarrow 2^\mathbb{Z}$ be a mapping.
If for each vertex $v$, $|L(v)|\ge d^+_D(v)+1$, then $G$ admits an $L$-coloring.
}\end{thm}
By considering balanced orientations of graphs, one can conclude the following corollary.
\begin{cor}\label{cor:bipartite}
{Let $G$ be a bipartite graph and $L:V(G)\rightarrow 2^\mathbb{Z}$ be a mapping.
If for each vertex $v$, 
$$|L(v)|\ge \lceil \frac{1}{2}d_G(v)\rceil+1,$$ then $G$ admits an $L$-coloring.
In addition, for a vertex $z$, we can have $|L(z)|\ge \lfloor \frac{1}{2}d_G(z)\rfloor+1$.
(Furthermore, if $G$ is $(k-1)$-edge-connected and contains $2k$ vertices with odd degrees, then we can replace the condition $|L(z)|\ge \lfloor \frac{1}{2}d_G(z)\rfloor+1$ for at least $k$ vertices $z$ with odd degrees.)
}\end{cor}
\begin{proof}
{Let $D$ be a balanced orientation of $G$ so that for each vertex $v$, $|d^+_D(v)-d^-_D(v)|\le 1$ which implies that $d^+_D(v)\le \lceil \frac{1}{2}d_G(v)\rceil$.
We may also assume that $d^+_D(z)\le \lfloor \frac{1}{2}d_G(z)\rfloor$.
Otherwise, we can reverse the orientation of $D$. Now, it is enough to apply Theorem~\ref{thm:Alon-Tarsi-1992} to complete
the first part of the proof. We shall below prove the remaining part.
Let $Z$ be a set of $k$ vertices with odd degrees and let $k-1$ edge-disjoint paths $P_1,\ldots, P_{k-1}$ starting from $k-1$ specified vertices in $Z$ and ended with some $k-1$ odd-degree vertices in $V(G)\setminus Z$ (using edge-version of Menger's Theorem). Let $G'$ be the graph obtained from $G$ by removing all edges of these paths. Now, we consider a balanced orientation for this graph such that for an arbitrary vertex $z$ with odd degree, its out-degree is less than its in-degree. Finally, we extend this orientation to $G$ by directing every path from end to begin.
This implies that $G$ has a balanced orientation such that for $k$ odd-degree vertices in $Z$ out-degree is less than in-degree
(Note that this method can also improve the needed edge-connectivity in Theorem 1 in~\cite{Katerinis-Tsikopoulos-2005} by one).
Again, it is enough to apply Theorem~\ref{thm:Alon-Tarsi-1992} to complete proof.
}\end{proof}
Motivated by Corollary~\ref{cor:bipartite} and Conjecture~\ref{conj:graph:epsilon:c}, 
we would like to pose the following problem.
\begin{prob}
{Prove or disprove: For every positive integer $k$, there exists a positive integer $c_k$ such that every $k$-colorable graph $G$ admits an $L$-coloring for every mapping $L:V(G)\rightarrow 2^\mathbb{Z}$ satisfying 
$|L(v)|\ge \frac{1}{2}(d_G(v)+c_k)$ for each vertex $v$.
}\end{prob}
Erd\H os, Rubin, and Taylor (1980) \cite{Erdos-Rubin-Taylor-1980} showed that the complete bipartite graph $K_{n,n}$ with $n=\binom{2k-1}{k}$ is not $k$-choosable. 
We observed that there is a subgraph with significantly smaller size and maximum degree which is still not $k$-choosable according to the following theorem. This construction introduces a non-$3$-choosable graph with size $43$.
It is known that the complete bipartite graph $K_{7,7}$ is not $3$-choosable (using lists obtained from Fano plane for each partite set); see~\cite[Page 129]{Erdos-Rubin-Taylor-1980}. By removing the edges of a matching of size $6$ from it, one can make another small non-$3$-choosable graph with size $49-6=43$ (both edges of the removed edges should have the same list).
\begin{thm}\label{thm:Kn-n:smaller-graph}
{For every integer $k$ with $k\ge 2$, there exists a non-$k$-choosable subgraph of the complete bipartite graph $K_{n,n}$ with $n=\binom{2k-1}{k}$ having $\sum^{k-1}_{i=0}\binom{k}{i}\binom{k+i}{k}$ edges and maximum degree $2^k-1$.
In particular, its vertex degrees lie in the following set:
$$\Big\{\sum_{j=i}^{k-1}\binom{k}{j}: 0\le i\le k-1\Big\}.$$
In addition, there exists a non-$k$-choosable subgraph of the complete bipartite graph $K_{n,n}$ having $\sum^{k-1}_{i=0}\binom{k}{i}\binom{k+i}{k}+k^2-1$ edges and maximum degree $2^k-2$ provided that $k\ge 3$.
}\end{thm}
\begin{proof}
{Let $(V_1, V_2) $ be a bipartition of $K_{n,n}$.
For every $v\in V_j$, we assign a unique subset $L(v)$ of $\{1,2,\ldots, 2k-1\}$ with size $k$ such that all of them are distinct, where $j=1,2$. Note that since $n=\binom{2k-1}{k}$, this assignment is possible.
 Let $H$ be the spanning subgraph of $K_{n, n}$ such that two vertices $v_1\in V_1$ and $v_2\in V_2$ are adjacent 
 if and only if for an integer $i$ with $i\in \{0,\ldots, k-1\}$, $L(v_1)\cup L(v_2)= \{1,\ldots, k+i\}$.
\begin{figure}[h]
 \centering
 \includegraphics[scale = 1.18]{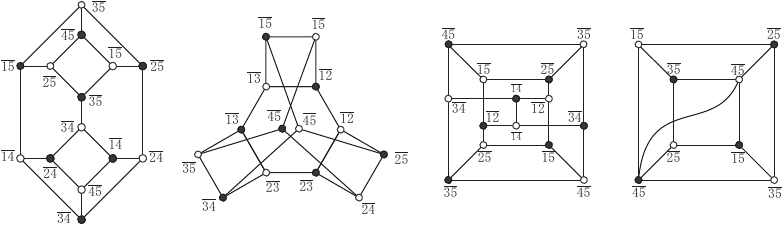}
 \caption{Four bipartite graphs which can be combined with $P(10,3)$ to make non-$3$-choosable bipartite graphs having either size $43$ and maximum degree $7$ or size $51$ and maximum degree $6$.
} 
\label{20BA-combination}
\end{figure}

For $k=2$, this graph is the same middle graph in Figure~\ref{smallest-non-chromatic-choosable-graph}, and for $k=3$, this graph is the union of the right graph in Figure~\ref{20BA-combination} and the left graph in Figure~\ref{20BA-combination-regular} (by keeping lists).
For every edge $e=v_1v_2$ with $L(v_1)\cup L(v_2)= \{1,\ldots, k+i\}$, the symmetric differences 
$L(v_1)\setminus L(v_2)$ and $L(v_2)\setminus L(v_1)$ have the same size $i$ and those can uniquely determine $L(v_1)$ and $L(v_2)$.
For selecting $L(v_1)\setminus L(v_2)$, we have $\binom{k+i}{i}$ choices and after specifying it, we have $\binom{k}{i}$ choices 
for selecting $L(v_2)\setminus L(v_1)$. Thus there are $\binom{k+i}{i}\binom{k}{i}$ edges of this type, and consequently the summation taken over all these numbers is the size of $G$.
Let $v$ be a vertex with $\max L(v)=k+i$. It is not difficult to check that this vertex is connected to 
$\binom{k}{i}$ vertices $u$ satisfying $\max L(u)\le k+i$.
In addition, it is connected to $\binom{k}{j}$ vertices $u$ satisfying $k+i<\max L(u)= k+j$.
Thus the degree of this vertex is $\sum_{j=i}^{k-1}\binom{k}{j}$.
We claim that $H$ does not have an $L$-coloring.
Suppose, to the contrary, that $H$ admits an $L$-coloring $c:V(G)\rightarrow \{1,\ldots, 2k-1\}$ such that for each vertex $v$, $c(v)\in L(v)$. Let us prove the following claim based on the graph construction. By applying this claim with setting $n=k-1$, there are two disjoint subsets $C_{1, n}$ and $C_{2, n}$ of $\{1,\ldots, k+n\}$ of size at least $n+1$. Therefore, $|C_{1, n}\cup C_{2, n}|\ge 2n+2=2k$ which is contradiction, as desired.

{\bf Claim A}: For every integer $n$ with $0\le n\le k-1$, there are two disjoint subsets $C_{1, n}$ and $C_{2, n}$ of $\{1,\ldots, k+n\}$ of size at least $n+1$ such that for every color $p\in C_{j, n}$, 
there is a vertex $v_j\in V_{j}$ with $p=c(v_j)\in L(v_j) = \{1,\ldots, k+n\}\setminus P_j$, where
$P_j$ is a subset of $C_{j, n}$ of size $n$.

{\bf Proof of Claim A}:
By induction on $n$. For $n=0$, we can set $C_{j, 0}=\{c(v_j)\}$, where $v_j$ is the unique vertex in $V_{j}$ with $L(v_j) = \{1,\ldots, k\}$. Since $v_1$ and $v_2$ are adjacent in $G$, their colors must be different and so these sets must be disjoint.
Assume that $n\ge 1$.
By the induction hypothesis, there are two disjoint subsets $C_{1, n}$ and $C_{2, n}$ of $\{1,\ldots, k+n-1\}$ satisfying the claim for $n-1$.
Consequently, there are $n$ colors $c_{j,1},\ldots, c_{j, n}$ in $C_{j, n-1}$.
We recursively define the sequence of colors $\{c_{j, m}\}^\infty_{m=1}$ such that $c_{j, m}=c(u_j)\in L(u_j)$, 
where 
$u_j$ is the unique vertex in $V_{j}$ with $L(u_j) = \{1,\ldots, k+n\}\setminus \{ c_{j, m-n},\ldots, c_{j, m-1}\}$ when $m> n$,
 and $u_j$ is a vertex in $ V_{j}$ with $c_{j, m}=c(u_j)\in L(u_j) = \{1,\ldots, k+n-1\}\setminus P_j$ for which
$P_j$ is a subset of $C_{j, n-1}$ of size $n-1$ when $m\le n$.
Let $S_j=C_{j, n-1}\cup \{c_{j, i}: i> n\}$.
We first show that $S_1$ and $S_2$ are disjoint.
Suppose, to the contrary, that $S_1\cap S_2\neq \emptyset$.
Since $C_{1, n-1}$ and $C_{2, n-1}$ are disjoint, we may assume that there is a color $c_{j, x_j}\in S_1\cap S_2$ with $x_j> n$.
Put $j'\in \{1,2\}\setminus \{j\}$.
We consider the integer $x_j$ with the smallest size. 
This means that $c_{j, i}\not \in S_{j'}$ for every $i$ with $i< x_j$.
Let $v_j$ be the unique vertex in $V_{j}$ with $L(v_j) = \{1,\ldots, k+n\}\setminus \{c_{j, x_j-n},\ldots, c_{j, x_j-1}\}$.
Clearly, we have the following two cases: 

{\bf Case 1}. $c_{j, x_j}=c_{j', x_{j'}}\in C_{j', n-1}$ for an integer $x_{j'}$ with $x_{j'}> n$.
In this case, we let $v_{j'}$ be the unique vertex in $V_{j'}$ with $L(v_{j'}) = \{1,\ldots, k+n\}\setminus \{c_{j', x_{j'}-n},\ldots, c_{j', x_{j'}-1}\}$. By the minimality property of $x_j$, color sets $\{c_{j, x_{j}-n},\ldots, c_{j, x_{j}-1}\}$ and $\{c_{j', x_{j'}-n},\ldots, c_{j', x_{j'}-1}\}$ must be disjoint. Consequently,
$L(v_1)\cup L(v_2)=\{1,\ldots, k+n\}$. This means that $v_1$ and $v_2$ are adjacent in $G$ and so their colors 
 $c_{j, x_j}$ and $c_{j', x_{j'}}$ must be different, which is a contradiction.

{\bf Case 2}. $c_{j, x_j} \in C_{j', n-1}$.
By the induction hypothesis, there is a subset of $P$ of $C_{j', n-1}$ of size $n-1$ satisfying
$c_{j, x_j}=c(v_{j'})\in L(v_{j'})$, where $v_{j'}$ is the unique vertex in $V_{j'}$ with $L(v_{j'}) = \{1,\ldots, k+n-1\}\setminus P$.
 By the minimality property of $x_j$, color sets $\{c_{j, x_{j}-n},\ldots, c_{j, x_{j}-1}\}$ and $C_{j', n-1}$ must be disjoint and so $\{c_{j, x_{j}-n},\ldots, c_{j, x_{j}-1}\}$ and $P$ are disjoint as well. 
Therefore, $L(v_1)\cup L(v_2)=\{1,\ldots, k+n-1\}$ when $k+n\in \{c_{j, x_{j}-n},\ldots, c_{j, x_{j}-1}\}$, 
and
$L(v_1)\cup L(v_2)=\{1,\ldots, k+n\}$ when $k+n\not\in \{c_{j, x_{j}-n},\ldots, c_{j, x_{j}-1}\}$, . This means that $v_1$ and $v_2$ are adjacent in $G$ and so their colors 
 $c_{j, x_j}$ and $c_{j', x_{j'}}$ must be different, which is again a contradiction. 
Consequently, in both cases $S_1$ and $S_2$ are disjoint.
Since the number of $n$-tuples are finite, there are two indices $a_j$ and $b_j$ with $n < a_j< b_j$ such that 
$(c_{j, a_j}, \ldots, c_{j, a_j+n-1})=(c_{j, b_j}, \ldots, c_{j, b_j+n-1})$. 
According to the definition of the sequence $\{c_{j, m}\}^\infty_{m=1}$, any color $c_{j, m}$ is different from any color $c_{j, m+t}$ with $0< t\le n$ which implies that $|b_j-a_j| \ge n+1$. 
On the other hand, one can inductively conclude that $c_{j, m}=c_{j, m+(b_j-a_j)}$ for all $m\ge a_j$.
Let us define $C_{j, n}=\{c_{j,m}: m \ge a_j\}$. Obviously, $|C_{j, n}|\ge n+1$. Since $S_1$ and $S_2$ are disjoint, 
$C_{1, n}$ and $C_{2, n}$ are disjoint as well.
In addition, for every color $c_{j, m}\in C_{j, n}$ with $m\ge b_j$, there is a subset $P_j=\{c_{j, m-n+1},\ldots, c_{j, m}\}$ of $C_{j, n}$ of size $n$ satisfying
$c_{j, m}=c(v_j)\in L(v_j)$, where $v_j$ is the unique vertex in $V_{j}$ with $L(v_j) = \{1,\ldots, k+n\}\setminus P_j$.
Hence the claim is proved. $\square$

To construct the required subgraph $H_0$ of the second assertion, we delete $k+1$ edges from $H$ and insert $k^2+k$ edges in
such a way that $\Delta(H_0)=\Delta(H)-(k+1)+k= 2^k-2$ and $|E(H_0)|=|E(H)|+k^2-1$. 
More precisely, we define $H_0$ be the spanning subgraph of $K_{n, n}$ such that two vertices $v_1\in V_1$ and $v_2\in V_2$ are adjacent 
 if and only if one of the following conditions holds:

\begin{enumerate}{

\item [$\bullet$]
For an integer $i\in \{2,\ldots, k-1\}$, $L(v_1)\cup L(v_2)= \{1,\ldots, k+i\}$.

\item [$\bullet$]
For an integer $c\in \{1,\ldots, k\}$, $L(v_{3-t})=(\{1,\ldots, k\}\cup \{k+t\})\setminus \{c\}$ for both integers $t\in \{1,2\}$.

\item [$\bullet$]
 For an integer $i\in \{1,2\}$, $L(v_1)\cup L(v_2)= \{1,\ldots,k\} \cup \{k+i\}$ provided that $L(v_i)\neq \{1,\ldots, k\}$.
}\end{enumerate}
For $k=3$, this graph is the union of the left graphs in Figures~\ref{20BA-combination} and~\ref{20BA-combination-regular}.
The proof of this part is similar to the above-mentioned arguments with minor modifications.
In fact, we first need to prove the following claims for the base step of our induction.
In these claims, $w_i$ denotes the unique vertex in $V_i$ with $L(w_i)=\{1,\ldots, k\}$. 

{\bf Claim B}: 
If $c(w_1)\neq c(w_2)$, then 
there are two disjoint color sets $C_{1}$ and $C_2$ of size at least $2$ 
satisfying  $c(w_{j'})\not \in C_{j}\subseteq \{1,\ldots, k\}\cup \{k+j\}$ such that for every color $p\in C_{j}$, there is a vertex $v_j\in V_{j}$ with $p=c(v_j)\in L(v_j) =(\{1,\ldots, k\}\cup \{k+j\})\setminus \{p_j\}$, 
where  $p_j \in C_{j}$ and $j'=3-j$.

{\bf Proof of Claim B}: Apply the same argument stated in the proof of Claim A for $n=1$ by setting $C_{j,0}=\{c(w_j)\}$. 
More precisely, we define $C_{j, 1}$ to be a subset of $\{1,\ldots, k\}\cup \{k+j\}$  such that
$u_j$ is the unique vertex in $V_{j}$ with $L(u_j) =(\{1,\ldots, k\}\cup \{k+j\}\setminus \{ c_{j, m-1}\})$ when $m> n$,
and $u_j=w_j$ when $m\le n$.
Note that the color $c(w_{j'})\in C_{j',0}$  must automatically be different from the color of all vertices $v\in V_j$ having  a list $L(v)\subseteq \{1,\ldots, k\}\cup \{k+j\}$, because $v$ and $w_{j'}$  are adjacent in $H_0$ when $v\neq w_j$. This is useful to prove the desired assertion  for Case 2.  $\square$

{\bf Claim C}: 
If $c(w_1)= c(w_2)$, for an integer $i\in \{1,2\}$, 
there are two disjoint color sets $C_{1}$ and $C_{2}$ of $\{1,\ldots, k\}\cup \{k+i\}$ of size at least $2$ 
satisfying $c(w_{j'})\not \in C_{j}$ or $k+j'\in C_{j}$ such that for every color $p\in C_{j}$, there is a vertex $v_j\in V_{j}$ with $p=c(v_j)\in L(v_j) =(\{1,\ldots, k\}\cup \{k+i\})\setminus \{p_j\}$, 
where  $p_j \in C_{j}$  and $j'=3-j$.

{\bf Proof of Claim C}:  Suppose, to the contrary, the claim is false.  Let $c=c(w_1)$.  
Since $v_1v_2\in E(H_0)$ for every $v_1\in V_1$ with $L(v_1)\subseteq \{1,\ldots, k+2\}$ and every $v_2\in V_2$ 
with  $L(v_2)\subseteq \{1,\ldots, k\} \cup \{k+1\}$, the color $c$ cannot appear on both sides of this edge.
This implies that there is an integer  $i\in \{1,2\}$ such that the color $c$ did not appear on all vertices $v\in V_{i'}$ satisfying 
$L(v)\subseteq \{1,\ldots, k\} \cup \{k+i\}$ and $v\neq w_{i'}$, where $i'=i$.
Let $Q_{j, i}$ be a directed graph with $V(Q_{j, i})=\{1,\ldots, k\}\cup \{k+i\}$ such that $xy$ is a directed edge from $x$ to $y$
 if $c(v)=y$, 
where $v$ is the unique vertex in $V_j$ with $L(v)=(\{1,\ldots, k\}\cup \{k+i\})\setminus \{x\}$.
The out-degree of every vertex in this digraph is precisely one. 
Thus there is a directed cycle $\mathcal{C}_{j, i}$ in $Q_{j, i}$ for every $j\in \{1,2\}$.
First assume that $k+i\not\in  V(\mathcal{C}_{i', i})$. This implies that $c\not\in  V(\mathcal{C}_{i',i})$, because the color $c$ did not appear on all vertices $v\in V_{i'}$ satisfying 
$L(v)\subseteq \{1,\ldots, k\} \cup \{k+i\}$ and $v\neq w_{i'}$.
 If $V(\mathcal{C}_{1, i})\cap V(\mathcal{C}_{2, i})=\emptyset$, then 
by setting $C_1=V(\mathcal{C}_{1, i})$ and $C_2=V(\mathcal{C}_{2, i})$, one can derive a contradiction.
 If $V(\mathcal{C}_{1, i})\cap V(\mathcal{C}_{2, i})\neq \emptyset$,
 then it is not difficult to check that 
$V(\mathcal{C}_{1, i})= V(\mathcal{C}_{2, i})$. 
Then $V(\mathcal{C}_{j, i})$ itself must be the vertex set of a component of $Q_{j, i}$ according to the edges of $H_0$.
Therefore,  there is a directed cycle $\mathcal{C}'_{i, i}$ in $Q_{i, i}\setminus V(\mathcal{C}_{i, i})$.
 Then by setting $C_{i'}=V(\mathcal{C}_{i', i})$ and $C_{i, i}=V(\mathcal{C}'_{i, i})$, one can again derive a contradiction.

Now,  assume that $k+i\in  V(\mathcal{C}_{i'})$ and so $c\in  V(\mathcal{C}_{i'})$.
 If $V(\mathcal{C}_{1, i})\cap V(\mathcal{C}_{2, i})=\emptyset$, 
then by setting $C_1=V(\mathcal{C}_{1, i})$ and $C_2=V(\mathcal{C}_{2, i})$, one can derive a contradiction.
 If $V(\mathcal{C}_{1, i})\cap V(\mathcal{C}_{2, i})\neq \emptyset$, then it is not difficult to check that 
$V(\mathcal{C}_{1, i})= V(\mathcal{C}_{2, i})$. 
Then $V(\mathcal{C}_{i, i})$ itself must be the vertex set of a component of $Q_{i,i}$ according to the edges of $H_0$. 
If $V(\mathcal{C}_{i, i})\neq V(Q_{i, i})$, then there is a directed cycle $\mathcal{C}'_{i, i}$ in $Q_{i, i}\setminus V(\mathcal{C}_{i, i})$.
By setting $C_1=V(\mathcal{C}_{1, i})$ and $C_2=V(\mathcal{C}'_{2, i})$, one can derive a contradiction.
So, suppose that $V(\mathcal{C}_{i'})= \{1,\ldots, k\} \cup \{k+i\}$. 
In this case, one can repeat the same arguments for $i'$ and conclude that there is a directed cycle $\mathcal{C}_{i, i'}$  of $Q_{i,i'}$ with $V(\mathcal{C}_{i, i'})=\{k+i', c\}$. This can lead us to drive a contradiction similarly, because 
 $V(\mathcal{C}_{i, i'})\neq  \{1,\ldots, k\} \cup \{k+i'\}$. Hence the claim holds. $\square$

In the next step, we define $C_{1,n-1}=C_1$ and $C_{2,n-1}=C_2$ to prove Claim A for $n=2$ in an analogous manner.
More precisely, the condition $c(w_{j})\not \in C_{j'}$ let us to avoid Case 2 when  $c_{j, x_j} =c(w_{j})$, 
and  the condition $k+j_0 \in C_j$ will not allow the vertex $w_j$ appears in vertices for constructing  sequence of colors $\{c_{j, m}\}^\infty_{m=1}$.
Finally, by an induction argument, one can analogously prove Claim~A for all $n\ge 3$ to derive a contradiction. 
Hence the proof is completed.
}\end{proof}
Motivated by Theorem~\ref{thm:Kn-n:smaller-graph}, we show below that $f(k, b) < (b-1) \sum_{0\le j\le\lceil\frac{k}{b-1}\rceil-1}\binom{k}{j}$ in the following theorem using a similar construction. 
\begin{thm}
{For any two integers $k$ and $b$ be two positive integers with $k\ge b\ge 2$, there is a $b$-colorable non-$k$-choosable graph $G$ satisfying 
$\Delta(G)=(b-1) \sum_{0\le j\le \lceil\frac{k}{b-1}\rceil-1}\binom{k}{j}$.
}\end{thm}
\begin{proof}
{Let $V_1,\ldots, V_b$ be $b$ disjoint sets of vertices such that all subsets of $\{1, \ldots, k+\lceil\frac{k}{b-1}\rceil-1\}$ of size $k$ are distributed on vertices of $V_i$. We denote this list assignment by $L$. Let $G$ be the graph with $V(G)=V_1\cup \cdots \cup V_b$ such that any two vertices $v_i\in V_i$ and $v_j\in V_j$ are adjacent if $i\neq j$ and $L(v_i)\cup L(v_j)=\{1,\ldots, k+t\}$ for an integer $t$ with $0\le t\le \lceil\frac{k}{b-1}\rceil-1 $. By an argument similar to the proof of Theorem~\ref{thm:Kn-n:smaller-graph}, one can show that $G$ is $b$-colorable but not $k$-choosable (using the list assignment of $L$). More precisely, one can prove the following claim by induction on $n$.
{\it Claim}: For every integer $n$ with $0\le n\le \lceil\frac{k}{b-1}\rceil-1$, there are $b$ disjoint subsets $C_{1, n},\ldots, C_{b, n}$ of $\{1,\ldots, k+n\}$ of size at least $n+1$ such that for every color $p\in C_{j, n}$, 
there is a vertex $v_j\in V_{j}$ with $p=c(v_j)\in L(v_j) = \{1,\ldots, k+n\}\setminus P_j$, where
$P_j$ is a subset of $C_{j, n}$ of size $n$.
By applying this claim with setting $n=\lceil\frac{k}{b-1}\rceil-1$, there are $b$ disjoint subsets $C_{1, n},\ldots, C_{b, n}$ of $\{1,\ldots, k+n\}$ of size at least $n+1$. Therefore, $|C_{1, n}\cup \cdots \cup C_{b, n}|\ge b(n+1)\ge k+n+1 $ which is contradiction, as desired.
Note that every vertex $v\in V_i$ with $\max L(v)=k+t$
is adjacent to $\binom{k}{t'}$ vertices $v'\in V_{i'}$ satisfying $\max L(v)=k+t'$ and $t< t'\le \lceil \frac{k}{b-1}\rceil -1$, where $i'\neq i$.
In addition, it is adjacent to $\binom{k}{t}$ vertices $v'\in V_{i'}$ satisfying $\max L(v)=k+t'$ and $0\le t' \le t$, where $i'\neq i$.
}\end{proof}
 For the special case $k=3$, by an innovative computer search, we observed that the graph introduced in Theorem~\ref{thm:Kn-n:smaller-graph} is the unique smallest bipartite graph of order $20$ which is not $2$-strongly $3$-choosable. In addition, we observed that there are $605$ edge-minimal bipartite graphs of order $20$ having maximum degree $6$ which are not $2$-strongly $3$-choosable, and among all of them three graphs have the smallest size $51$. 
\begin{observ}\label{observ:P(10, 3)}
{There is a bipartite graph $G$ of order $20$ having maximum degree $7$ and size $43$ which is not $2$-strongly $3$-choosable.
Moreover, there are three bipartite graphs $G$ of order $20$ having maximum degree $6$ and size $51$ which are not $2$-strongly $3$-choosable. 
}\end{observ}
\begin{proof}
{Let $\mathbb{Z}_5$ be the set of five colors, and let $S$ be the set of all subsets of $\mathbb{Z}_5$ having exactly $3$ colors. Note that $|S|=10$. For every $L\in S$, we consider two vertices $v$ and $v'$ with the same list $L$. In addition, we add an edge between $v$ and $v'$, if those two lists corresponding them have exactly one color in common. 
Let $G$ be the resulting bipartite graph. Note that $G$ is isomorphic to the generalized Petersen graph $P(10,3)$.
We can combine this graph with a graph illustrated in Figure~\ref{20BA-combination} such that lists of identified vertices would be the same (we denoted by $\overline{ij}$ the list $\mathbb{Z}_5\setminus \{i,j\}$). By a computer search, we observed that every such a graph does not admit an $L$-coloring (using only colors from $\mathbb{Z}_5$).
}\end{proof}

 It is not difficult to verify that if a subgraph of $K_{n,n}$ is not $(k-1)$-strongly $k$-choosable, then this subgraph must contain a copy of the middle layer graph with the parameter $k$. (For the special case $k=3$, this middle layer graph is the same generalized Petersen graph $P(10, 3)$). By a computer search on the specified outputs of the program {\it genreg} due to Meringer~\cite{Meringer-1999}), we observed that there exists a special non-$3$-choosable $6$-regular bipartite graph of order $20$ having two edge-disjoint copies of the generalized Petersen graph $P(10, 3)$; see Figure~\ref{20BA-combination-regular}. Motivated by this graph, we would like to pose the following problem.
\begin{prob}
{Let $k$ be a positive integer with $k\ge 4$. How many copies of the middle layer graph with the parameter $k$ are enough to make a bipartite graph of order $2\binom{2k-1}{k}$ which is not $(k-1)$-strongly $k$-choosable? 
}\end{prob}
\begin{figure}[h]
 \centering
 \includegraphics[scale = 1.15]{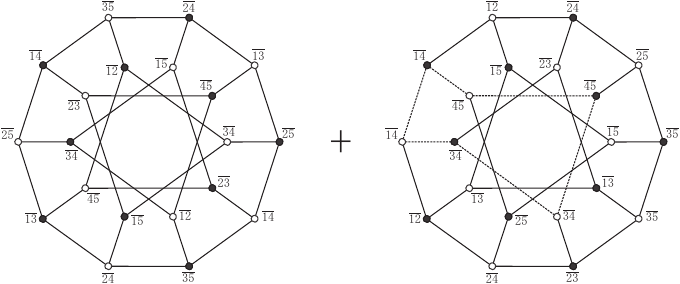}
 \caption{A non-$3$-choosable $6$-regular bipartite graph of order $20$ which can be edge-decomposed into two isomorphic copies of the generalised Petersen graph $P(10,3)$ (six dotted edge can be deleted to make a smaller non-$3$-choosable graph).} 
\label{20BA-combination-regular}
\end{figure}

In 2002 Bessy, Havet, and Palaysi~\cite{Bessy-Havet-Palaysi-2002} constructed a non-$3$-choosable bipartite graph $G$ with maximum degree $6$ having only $128$ vertices (the graph illustrated in Figure~\ref{20BA-combination-regular} shows that this number can be reduced to $20$). They also constructed a non-$3$-choosable bipartite graph $G$ with maximum degree $5$ having $846$ vertices. We observed that this order can be reduced to $390$. More precisely, we only need to replace the left and right graphs in Figure~\ref{fig:Bipartite-D5} instead of the graphs in Figures~2 and~3 in \cite{Bessy-Havet-Palaysi-2002} and repeat the same procedure of the proof (furthermore, this order can be pushed down to $376$ by modifying some lists to identify more pairs of vertices with degree $3$).
\begin{observ}
{There exists a non-$3$-choosable bipartite graphs $G$ of order $390$ with maximum degree $5$.
}\end{observ}
\begin{figure}[h]
 \centering
 \includegraphics[scale = 1.7]{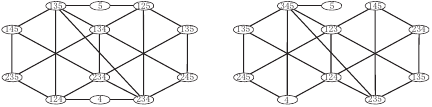}
 \caption{Two almost non-$3$-choosable bipartite graphs without given list-colorings.} 
\label{fig:Bipartite-D5}
\end{figure}
%
%
%
\subsection{Planar graphs}
It is known that every planar graph is $4$-colorable~\cite{Appel-Haken-1976}. 
We feel that Conjecture~\ref{conj:simpler} holds for planar graphs with the following stronger version. 
\begin{conj}
{Every planar graph with maximum degree at most $6$ is $4$-choosable.
}\end{conj}
One may ask whether this conjecture holds by replacing a bit larger upper bound on the maxim degree.
So, we would like to pose the following question.
\begin{prob}\label{prob:k}
{What is the maximum number $k$ such that every planar graph $G$ satisfying $\Delta(G)\le k$ is $4$-choosable? 
Is there a $5$-connected non-$4$-choosable planar graph? (or even without separating $3$-cycles and $4$-cycles)?
If yes, what is the minimum value of maximum degrees?
}\end{prob}
In 1993 Voigt \cite{Voigt-1993} constructed a non-$4$-choosable planar graph of order $238$ with maximum degree $38$. 
Later, Mirzakhani (1996)~\cite{Mirzakhani-1996} reduced the order to $63$ but by increasing the maximum degree to $42$.
At about the same time, Gutner (1996) \cite{Gutner-1996} introduced another simple construction using $75$ vertices for which two vertices have degree $48$. These bounds on the maximum degree can be pushed further down as the following theorem. Note that the construction of Mirzakhani used five colors in list assignments and Voigt and Wirth~(1997)~\cite{Voigt-Wirth-1997} improved the list assignments of Gutner~\cite{Gutner-1996} to this version.
\begin{thm}
{There exists a $4$-colorable (resp. $3$-colorable) planar graph with maximum degree $8$ (resp. $10$) which is not $4$-choosable.
}\end{thm}
\begin{proof}
{There is a $3$-colorable planar (near triangulation) graph of order $17$ with given lists on vertices so that for every list coloring of it, the color $5$ appears on at least one vertex of the outer face, see~\cite{Mirzakhani-1996}. 
This graph is illustrated in Figure~\ref{fig:Block} (left graph) and its outer face has size $12$. 
We observed that it is possible to reduce this size to $4$ but by increasing the order and maximum degree to $29$ and $12$; see the right graph in Figure~\ref{fig:Block}. Their proofs are based in the following simple and essential claim.
Indeed, this claim can help us to find pairs of non-adjacent vertices $v$ and $u$ such that the color of one vertex is forbidden by the list of the other vertex which implies that their colors must be different. This allows us to add the edge $uv$ by keeping the list coloring and finally fining a clique of order $5$ or $6$, which derives a contradiction based on the number of used colors in these cliques. We shall below examine the strategy of the proof for the third graph. 

{\bf Claim}.
If $C$ is a graph of order $5$ consists of a vertex $v_0$ having the list $\mathbb{Z}_4$ adjacent to all  other vertices and 
a four cycle $v_1v_2v_3v_4$ having lists $\mathbb{Z}_4\setminus \{1\},\ldots, \mathbb{Z}_4\setminus \{4\}$, respectively, then for every such a list coloring, there is an integer $i\in \{1,\ldots, 4\}$ such that $v_{i}$ and $v_{i+1}$ are colored with $i-1$ and $i+2$ (mod $4$), respectively, where $\mathbb{Z}_4=\{1,2,3,4\}$.

Let $P$ be the third graph in Figure~\ref{fig:Block} with given lists on vertices, where $\overline{i}=\mathbb{Z}_5\setminus \{i\}$. (This graph is $3$-colorable and has maximum degree at most $10$). 
We claim that for every list coloring of $P$, the color $5$ appears on at least one black vertex of the outer face. 
Suppose, to the contrary, that $P$ admits a list coloring $c:V(G)\rightarrow \mathbb{Z}_5$ such that the color $5$ does not appear on all black vertices of the outer face. 
Let $C$ be the induced subgraph of $P$ having all black vertices $v_1,\ldots, v_4$ and the vertex $v_0$ adjacent to all of them. Assume that $v_i$ has the list $\overline{i}$. By the above-mentioned claim, there is an integer $i\in \{1,\ldots, 4\}$ such that $v_{i}$ and $v_{i+1}$ are colored with $i-1$ and $i+2$, respectively. 
Consider the four cycle $v_{i} v_{i+1}w_{i-1} w_{i+2}$ having lists 
 $\overline{i},\overline{i+1}, \overline{i-1}$, and $\overline{i+2}$, respectively. 
Let $v'_0$ and $w'_0$ be the inner and outer vertices adjacent to all vertices of this cycle.
Note that the lists of $v'_0$ and $w'_0$ are $\overline{i}$ and $\overline{i+1}$.
Since the color of $v_{i}$ is forbidden by the list of $w_{i-1}$, their colors are different and so we can keep this list coloring $c$ by adding the edge $v_{i}w_{i-1}$. Likewise, we can keep this list coloring $c$ by adding the edge $v_{i+1}w_{i+2}$.
If we can also keep this list coloring $c$ by adding the edge $v'_0w'_0$, then the induced subgraph with vertices $\{v_{i}, v_{i+1},w_{i-1} ,w_{i+2},v'_0,w'_0\}$ forms a clique of order $6$ and so we have at least six colors, which is a contradiction.
Therefore, both vertices $v'_0$ and $w'_0$ receive the same color which must be different from the colors in $\{i, i+1\}$ (because of the forbidden lists) and $\{i-1,i+2\}$ (they are adjacent to $v_i$ and $v_{i+1}$).
This means that both vertices $v'_0$ and $w'_0$ must receive the same color $5$.
(Now, we have shown that why the color $5$ appears on at least one vertex of the outer face of the second graph in Figure~\ref{fig:Block}). 
Therefore, $w_{i-1}$ and $w_{i+2}$ must receive both colors $i$ and $i+1$ 
(because of their forbidden lists and adjacency to $v_{i}$ and $v_{i+1}$). 
Thus we have either $c(w_{i-1})=i$ and $c(w_{i+2})=i+1$ or $c(w_{i-1})=i+1$ and $c(w_{i+2})=i$.
In either case, there is an induced subgraph of $6$ minus a matching $M$ of size three such that 
both vertices of a every edge $e\in E(M)$ must receive different colors (because the color of one vertex is forbidden by the list of the other vertex). 
This means that after inserting $M$, we can keep the list coloring in this subgraph.
 Again, we come up a clique of order $6$ and so we have at least six colors, which is a contradiction. Consequently, the claim holds.
(Note that the proof of all graphs in Figure~\ref{fig:Block} are similar in order to make cliques of size $6$, except the first graph for which we should create a clique of order $5$ with vertices $v_{i}, v_{i+1},w_{i-1} ,w_{i+2},v'_0$
 after adding two edges $v_{i}w_{i-1}, v_{i+1}w_{i+2}$ but all its vertices receive colors only from $\mathbb{Z}_4$).

Let $H$ be the $3$-colorable planar $4$-regular illustrated in Figure~\ref{4-regular} (right graph) with given lists on vertices and also for each $v\in V(H)$, we insert a copy of $P$ and join $v$ to all four black vertices on the outer face of that copy $P_v$. It is easy to check that the resulting graph $G$ is $3$-colorable but not $4$-choosable (using the same lists of the figures). In addition, it has maximum degree $10$.

\begin{figure}[h]
 \centering
 \includegraphics[scale = 1.2]{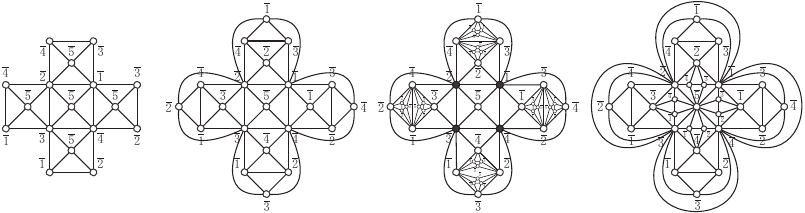}
 \caption{Four planar graphs $G$ that any list coloring of them contains the color $5$ on at least one vertex of the outer face $F$ (in particular, black vertices for the third graph), where 
$|F|\in \{12,8,4\}$ and $\Delta(G)\in \{7,9,12\}$.} 
\label{fig:Block}
\end{figure}

Now, we are going to push down the upper bound on the maximum degree for $4$-chromatic planar graphs.
It is not difficult to check that the graph illustrated in Figure~\ref{fig:maximum-8-planar} (left part) 
does not admit a list coloring with respect to the given lists, where $\bar{i}=\mathbb{Z}_5\setminus \{i\}$.
Let $v_i$ the vertex with the list $\{i\}$, where $i=1,2$. 
Let $L_{1,2}$ be the graph obtained from two copies of the graph illustrated in Figure~\ref{fig:maximum-8-planar} (right graph) 
with the same lists by identifying the first copy of $v_1$ and the second copy of $v_2$ into a single vertex $a_1$ and
by identifying the second copy of $v_1$ and the first copy of $v_2$ into a single vertex $a_2$. Let $b_1$ be the black vertex of the first copy and let $b_2$ be the black vertex of the second copy.
Let $H_{x,y}$ be the graph obtained from $L_{1,2}$ by distributing the same list $\bar{5}$ on both vertices $a_1$ and $a_2$, and also by 
permuting colors $\{1,2,3,4\}$ such that the color $1$ is replaced by $x$ and the color $2$ is replaced by $y$.
It is not difficult to check that the graph $H_{x, y}+b_1b_2$ does not admit such a list coloring but by replacing the list $\{x,y\}$ on both vertices $a_1$ and $a_2$.

Now, we define the graph obtained from $H_{x,y}$ by using the pattern illustrated in Figure~\ref{fig:maximum-8-planar} (left part) by replacing six graphs $H_{xy}$ in the specified parts so that copies of both black vertices $b_1$ and $b_2$ overlap black vertices of the pattern, where $x,y\in \bar{5}$. Suppose, to the contrary, this graph admits such a list coloring. Let $u_1$ and $u_2$ be the two adjacent vertices in outer cycle of the left graph in Figure~\ref{fig:maximum-8-planar}. 
Assume that these vertices receive different colors $i$ and $j$ from $\bar{5}$.
Then the two black vertices of $H_{i',j'}$ must receive both colors $i$ and $j$ which implies that two copies of $a_1$ and $a_2$ in $H_{i',j'}$ must receive both colors $i'$ and $j'$, where $i', j'\in \bar{5}\setminus \{i, j\}$. This is a contradiction to the property of $H_{i',j'}$. Hence the proof is completed.
\begin{figure}[h]
 \centering
 \includegraphics[scale =0.74]{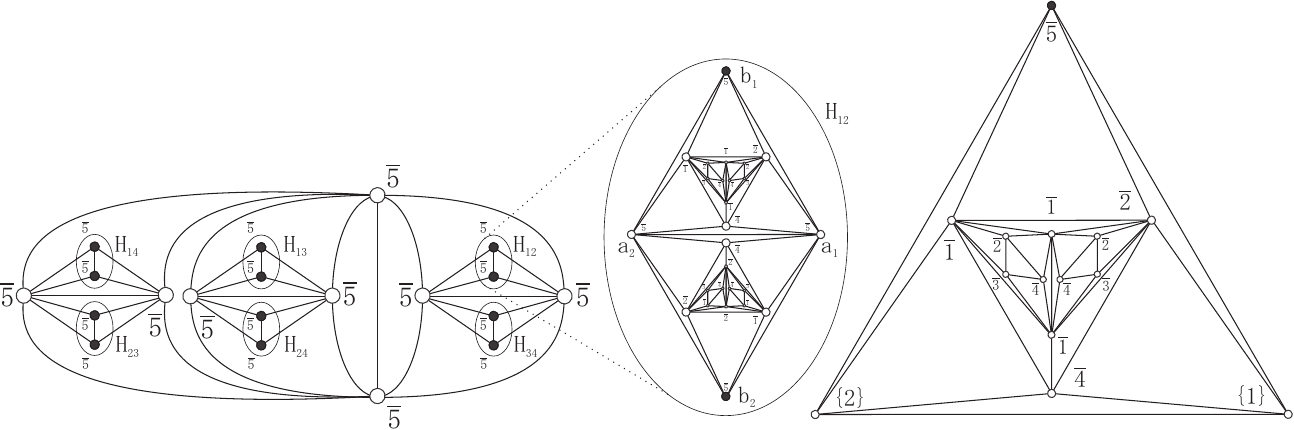}
 \caption{An almost non-$4$-choosable triangulation without a given list-coloring (right), and the frame for making a non-$4$-choosable planar graph with maximum degree $8$ (left).} 
\label{fig:maximum-8-planar}
\end{figure}
}\end{proof}
In 1995 Voigt \cite{Voigt-1995} constructed a non-$3$-choosable planar triangle-free graph of order $166$ with maximum degree $27$. Later, Gutner (1996) \cite{Gutner-1996} improved its order to $164$ and its maximum degree to $18$.
Finally, Glebov, Kostochka, Tashkinov (2005) \cite{Glebov-Kostochka-Tashkinov-2005} constructed a non-$3$-choosable planar triangle-free graph of order $97$ having maximum degree $10$. This bound on the maximum degree can be pushed further down as the following theorem.
\begin{thm}
{There exists a non-$3$-choosable planar triangle-free graph with maximum degree $7$.
}\end{thm}
\begin{proof}
{There exists a planar triangle-free graph $H_z$ with maximum degree $7$ without a list coloring $L$ such that $|L(v)|=3$ for all vertices $v$, except for a vertex $z$ with degree $5$, see \cite{Glebov-Kostochka-Tashkinov-2005}. In particular, $L(z)=\{1\}$. We consider a new five cycle $C_5$ and for each $v\in V(C)$, we select a copy of the graph $H_{v'}$ with the same list assignment by identifying any two vertices $v\in V(C)$ and $v'$ into the same vertex. In addition, we distribute the same list $\{1,2,3\}$ on all vertices of this cycle. Obviously, the resulting planar triangle-free graph has maximum degree at most $7$ and does not admit such a list coloring and so it must not be $3$-choosable.
}\end{proof}
We observed that there are at least $306$ non-$3$-choosable planar triangle-free graphs of order $97$ (using only five colors in lists). In fact, it is enough to replace two graphs with given list assignment of Figure~\ref{fig:Tri-Free} in the graph construction in~\cite{Glebov-Kostochka-Tashkinov-2005}. We understood these two graphs are smallest ones with a desired property by using outputs of the program {\it plantri} due to Brinkmann and McKay~\cite{Brinkmann-McKay-2007}. (The second one also appeared in~\cite[Figure 2]{Montassier-2006} with the same list assignment and the underline graph was originally introduced by Voigt (1995) \cite[Figure 2]{Voigt-1995}).
Finally, we pose the following problem for further investigation.

\begin{figure}[h]
 \centering
 \includegraphics[scale = 1.15]{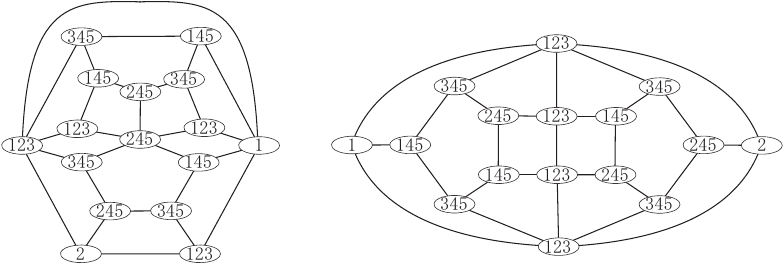}
 \caption{Two almost non-$3$-choosable triangle-free planar graphs of order $16$ without given list-colorings (using five colors).} 
\label{fig:Tri-Free}
\end{figure}

\begin{prob}
{Is there a planar triangle-free graph which is not $1$-strongly $3$-choosable? If yes, what is the minimum value of maximum degrees?
}\end{prob}
\section{Acknowledgements}
The author would like to appreciate Masaki Kashima for carefully checking proofs of the paper and finding some mistakes.
%
%
%
%
%
%
%
%
%

\end{document}